\documentclass[a4paper]{article}

\usepackage[utf8]{inputenc}

\usepackage[top=1in,bottom=1in,left=1.25in,right=1.25in]{geometry}
\usepackage[font=small,skip=0pt]{caption}

\usepackage{stmaryrd}
\usepackage{multirow}
\usepackage{float}
\usepackage{diagbox}
\usepackage{tabularx}
\usepackage{multirow}
\usepackage{graphicx}
\usepackage{enumitem}

\usepackage[all]{xy}

\usepackage{xcolor,palatino}
\definecolor{ultramarine}{RGB}{0,32,96}
\colorlet{mygreen}{green!20!gray}
\colorlet{myultramarine}{ultramarine!20!gray}

\usepackage{amsmath}
\usepackage{amsthm}
\usepackage{amssymb}
\usepackage{mathrsfs} 
\usepackage{mathtools}
\mathtoolsset{showonlyrefs}
\usepackage{relsize}
\usepackage{exscale}

\usepackage{tikz}
\usetikzlibrary{
	arrows,
	arrows.meta,
	cd,
	decorations, 
	decorations.markings,
	external,
	fit, 
	intersections,
	matrix,
	decorations.pathreplacing, 
	calc, 
	positioning,
	shapes.geometric,
	trees
}

\usepackage[colorlinks=true,citecolor=mygreen, linkcolor=mygreen, urlcolor=myultramarine,breaklinks, naturalnames]{hyperref}

\usepackage[backrefs]{amsrefs}

\numberwithin{equation}{section}
\numberwithin{table}{section}

\allowdisplaybreaks[4]

\DeclareFontFamily{U}{BOONDOX-calo}{\skewchar\font=45 }
\DeclareFontShape{U}{BOONDOX-calo}{m}{n}{
   <-> s*[1.05] BOONDOX-r-calo}{}
\DeclareFontShape{U}{BOONDOX-calo}{b}{n}{
   <-> s*[1.05] BOONDOX-b-calo}{}
\DeclareMathAlphabet{\mathcalboondox}{U}{BOONDOX-calo}{m}{n}
\SetMathAlphabet{\mathcalboondox}{bold}{U}{BOONDOX-calo}{b}{n}
\DeclareMathAlphabet{\mathbcalboondox}{U}{BOONDOX-calo}{b}{n}

\newtheorem{thm}{Theorem}[section]
\newtheorem{adef}[thm]{Definition}
\newtheorem{rk}[thm]{Remark}
\newtheorem{lem}[thm]{Lemma}

\newtheorem{prop}[thm]{Proposition}
\newtheorem{cor}[thm]{Corollary}
\newtheorem{fact}[thm]{Fact}
\newtheorem*{prop*}{Proposition}

\newcommand\FK{{\operatorname{FK}}}
\newcommand\id{{\operatorname{id}}}

\newcommand\Hom{{\operatorname{Hom}}}
\newcommand\Homm{{\mathcal{H}om}}
\newcommand\HH{{\operatorname{HH}}}
\newcommand\Z{{\mathcal{Z}}}

\newcommand\T{{\mathbb{T}}}
\newcommand\NN{{\mathbb{N}}}
\newcommand\RR{{\mathbb{R}}}
\newcommand\ZZ{{\mathbb{Z}}}

\title{\texorpdfstring{Gerstenhaber structure on Hochschild cohomology of the Fomin-Kirillov algebra on $3$ generators}{Gerstenhaber structure on Hochschild cohomology of the Fomin-Kirillov algebra on 3 generators}}
\author{Estanislao Herscovich and Ziling Li}
\date{}

\begin{document}

\maketitle

\begin{abstract}
The goal of this article is to compute the Gerstenhaber bracket of the Hochschild cohomology of the Fomin-Kirillov algebra on three generators over a field of characteristic different from $2$ and $3$. 
This is in part based on a general method we introduce to easily compute the Gerstenhaber bracket between elements of $\HH^{0}(A)$ and elements of $\HH^{n}(A)$ for $n \in \NN_{0}$, the method by M. Suárez-Álvarez in \cite{Mariano} to calculate the Gerstenhaber bracket between elements of $\HH^{1}(A)$ and elements of $\HH^{n}(A)$ for any $n \in \NN_{0}$, as well as an elementary result that allows to compute the remaining brackets from the previous ones. 
We also show that the Gerstenhaber bracket of $\HH^{\bullet}(A)$ is not induced by any Batalin-Vilkovisky generator. 
\end{abstract}




\section{\texorpdfstring{Introduction}{Introduction}}
\label{Introduction}

The goal of this article is to completely compute the Gerstenhaber bracket of the Hochschild cohomology $\HH^{\bullet}(A)$ of the Fomin-Kirillov algebra $A$ on three generators over a field $\Bbbk$ of characteristic different from $2$ and $3$. 
In the paper \cite{es zl} we have computed the algebra structure on $\HH^{\bullet}(A)$, and we showed that it is a finitely generated graded algebra with a minimal set $\{ X_{i} | i \in \llbracket 1 , 14 \rrbracket \}$ of $14$ generators, three of which are in $\HH^{0}(A)$, five are in $\HH^{1}(A)$, four are in $\HH^{2}(A)$, one in $\HH^{3}(A)$ and one in $\HH^{4}(A)$. 
In this article we compute the Gerstenhaber bracket between these generators, which fully determines the Gerstenhaber bracket on $\HH^{\bullet}(A)$. 

Our calculations are organised as follows. 
First we provide a general method of homological flavour to easily compute the Gerstenhaber bracket between elements of $\HH^{0}(A)$ and elements of $\HH^{n}(A)$ for any $n \in \NN_{0}$ and any algebra $A$ over a field $\Bbbk$ (see Theorem \ref{HH0}). 
We then use this method for the particular case of the Fomin-Kirillov algebra $A$ on three generators over a field of characteristic different from $2$ and $3$ (see Proposition \ref{prop:bracket-H0-Xi}). 
Secondly, for $A$ as before, we compute the Gerstenhaber bracket between elements of $\HH^{1}(A)$ and elements of $\HH^{n}(A)$ for any $n \in \NN_{0}$ using the method introduced by M. Suárez-Álvarez in \cite{Mariano} (see Propositions \ref{prop:X_8} and \ref{prop:X4}). 
Finally, we present a simple result that allows us to compute the remaining Gerstenhaber brackets under some assumptions on the algebra structure of the Hochschild cohomology of an algebra (see Lemma \ref{lemma:tech}), which are verified in the case of the Fomin-Kirillov algebra $A$ on three generators over a field of characteristic different from $2$ and $3$ (see Proposition \ref{prop:remaining brackets}). 
We summarize our results in Table \ref{table:Gerstenhaber brackets}. 
By using the explicit expression of the Gerstenhaber bracket of $\HH^{\bullet}(A)$ together with a direct computation, we show that the former is not induced by any Batalin-Vilkovisky generator (see Proposition \ref{prop:no generator}). 

The article is organised as follows. 
In the first subsection of Section \ref{section:Gerstenhaber brackets} we recall some general facts about the Gerstenhaber bracket on Hochschild cohomology. 
In Subsection \ref{subsection:Gerstenhaber brackets of HH^0} we present our general method to compute the Gerstenhaber bracket between elements of $\HH^{0}(A)$ and elements of $\HH^{n}(A)$ for any $n \in \NN_{0}$, whereas in Subsection \ref{subsection:Gerstenhaber brackets of HH^1} we recall the method of \cite{Mariano} to compute the Gerstenhaber bracket between elements of $\HH^{1}(A)$ and elements of $\HH^{n}(A)$ for any $n \in \NN_{0}$. 
In Section \ref{sec:basics} we recall the definition of the Fomin-Kirillov algebra on three generators, and we summarize 
some results of \cite{es zl} about its Hochschild cohomology, in particular those we are going to use in the sequel.
Finally, in Section \ref{section:Gerstenhaber brackets on Hochschild cohomology} we compute the Gerstenhaber brackets of the generators of the algebra structure of the Hochschild cohomology of the Fomin-Kirillov algebra $A$ on three generators, following the path described in the previous paragraph.  

In the whole article, 
we will denote by $\NN$ (resp., $\NN_{0}$) the set of positive (resp., nonnegative) integers, and $\ZZ$ the set of integers.  
Moreover, given $i, j \in \ZZ$ we define the integer interval $\llbracket i , j \rrbracket = \{ n \in \ZZ | i \leqslant n \leqslant j \}$. 
To reduce space in the expressions of the article we will typically denote the composition $f\circ g$ of maps $f$ and $g$ simply by their juxtaposition $f g$. 
For a field $\Bbbk$, all maps between $\Bbbk$-vector spaces will be $\Bbbk$-linear and all unadorned tensor products $\otimes$ 
will be over $\Bbbk$. 

\section{\texorpdfstring{The Gerstenhaber bracket on Hochschild cohomology}{The Gerstenhaber bracket on Hochschild cohomology}}
\label{section:Gerstenhaber brackets} 

All along this section we will consider $\Bbbk$ to be a field and $A$ to be a (unital associative) $\Bbbk$-algebra. 

\subsection{Generalities on the Gerstenhaber bracket}
\label{subsection:Generalities Gerstenhaber brackets}

In this subsection we recall several basic definitions and results concerning the Gerstenhaber bracket, that we will utilize in the sequel. 

Recall that the \textbf{\textcolor{ultramarine}{bar resolution}} $(B_{\bullet}(A), d_{\bullet})$ of $A$ is given by 
$B_n(A)=A^{\otimes (n+2)}$ for $n\in \NN_0$, with the differentials $d_n: B_{n}(A) \to B_{n-1}(A)$ given by 
\[ d_n(a_0| \dots |a_{n+1}) =\sum_{j=0}^n (-1)^{j}a_0|\dots |a_{j-1}|a_ja_{j+1}|a_{j+2}|\dots|a_{n+1}  \] 
for $a_0,\dots,a_{n+1}\in A$ and $n\in\NN$, and the augmentation $\pi: B_0(A)=A\otimes A\to A$ defined by the multiplication of $A$. 
We will typically write $a_0 | \dots |a_{n+1}$ instead of $a_0 \otimes \dots \otimes a_{n+1}$ for simplicity.
There is an isomorphism 
\begin{equation}
\label{eq:F-bar}
     F : \Hom_{A^e}\big(B_n(A), A\big)\longrightarrow \Hom_{\Bbbk}(A^{\otimes n},A)
\end{equation}
given by $F(f)(a_1|\dots|a_n)=f(1|a_1|\dots|a_n|1)$ for $f\in \Hom_{A^e}(B_n(A), A)$ and $a_1,\dots,a_n\in A$.
The inverse map 
\begin{equation}
\label{eq:G-bar}
     G : \Hom_{\Bbbk}(A^{\otimes n},A) \longrightarrow  
     \Hom_{A^e}\big(B_n(A), A\big)
\end{equation}
of $F$ is explicitly given by $G(g)(a_0|\dots|a_{n+1})=a_0g(a_1|\dots|a_n)a_{n+1}$ for $g\in \Hom_{\Bbbk}(A^{\otimes n},A)$ and $a_0,\dots,a_{n+1}\in A$. 

The following definition is classical (see for instance \cite{Sarah}, Def. 1.4.1).
\begin{adef}
\label{Gerstenhaber bracket at chain level}
Let $m,n\in\NN_0$, $f\in\Hom_{\Bbbk}(A^{\otimes m},A)$ and $g\in \Hom_{\Bbbk}(A^{\otimes n},A)$. 
The {\color{ultramarine}{\textbf{Gerstenhaber bracket}}} $[f,g]$ is defined at the chain level as the element of $\Hom_{\Bbbk}(A^{\otimes (m+n-1)},A)$ given by 
\[  [f,g]=f\circ_{G} g-(-1)^{(m-1)(n-1)}g\circ_{G} f , \]
where $f\circ_{G} g$ is defined by 
\[ (f\circ_{G} g)(a_1|\dots |a_{m+n-1})=\sum_{i=1}^{m} (-1)^{(n-1)(i-1)}f\big(a_1|\dots |a_{i-1}|g(a_i|\dots|a_{i+n-1})|a_{i+n}|\dots|a_{m+n-1}\big). \]
Moreover, if $m=0$, then $f\circ_{G} g=0$, while if $n=0$, then the formula should be interpreted by taking the value $g(1)$ in place of $g(a_i|\dots|a_{i+n-1})$.
\end{adef}

Using the isomorphisms $F$ and $G$ of chain complexes given above, one defines the {\color{ultramarine}{\textbf{Gerstenhaber bracket}}}
in $\Hom_{A^e}(B_{\bullet}(A),A)$ by $[f,g]=G([F(f),F(g)])\in \Hom_{A^e}(B_{m+n-1}(A),A)$ for $f\in \Hom_{A^e}(B_{m}(A),A)$, $g\in \Hom_{A^e}(B_n(A),A)$ and $m,n\in \NN_0$. 
The Gerstenhaber bracket given before induces a well-defined bilinear map  
\begin{equation}
\label{eq:gerstenhaber}
    [ \hskip 0.6mm , ] : \operatorname{H}^{m}\big(\Hom_{A^e}(B_{\bullet}(A),A)\big) \times \operatorname{H}^{n}\big(\Hom_{A^e}(B_{\bullet}(A),A)\big) \rightarrow \operatorname{H}^{m+n-1}\big(\Hom_{A^e}(B_{\bullet}(A),A)\big)
\end{equation}
for all $m, n \in \NN_{0}$, that we also call the 
{\color{ultramarine}{\textbf{Gerstenhaber bracket}}}. 

More generally, let $(P_{\bullet},\partial_{\bullet})$ be a projective bimodule resolution over $A$ with augmentation $\mu:P_0\to A$.
Let $i_{\bullet}:P_{\bullet}\to B_{\bullet}(A)$ and $p_{\bullet}:B_{\bullet}(A)\to P_{\bullet}$ be morphisms of complexes of $A$-bimodules lifting $\id_A$, so $p_{\bullet} i_{\bullet}$ is homotopic to $\id_{P_{\bullet}}$ and $i_{\bullet} p_{\bullet}$ is homotopic to $\id_{B_{\bullet}(A)}$. 
We also recall that the morphisms $i_{\bullet}$ and $p_{\bullet}$ induce the quasi-isomorphisms
$i_{\bullet}^{*}:\Hom_{A^e}(B_{\bullet}(A),A)\to \Hom_{A^e}(P_{\bullet},A) $ 
and 
$p^{*}_{\bullet}:\Hom_{A^e}(P_{\bullet},A)\to \Hom_{A^e}(B_{\bullet}(A),A)$ given by $i_{\bullet}^{*}(f) = f i_{\bullet}$ and 
$p_{\bullet}^{*}(g) = g p_{\bullet}$ for $f\in \Hom_{A^e}(B_{\bullet}(A),A)$ and $g\in \Hom_{A^e}(P_{\bullet},A)$, respectively.
Moreover, $\operatorname{H}(i^{*}_{\bullet}),\operatorname{H}(p^{*}_{\bullet}):\HH^{\bullet}(A)\to \HH^{\bullet}(A) $ are independent of the choice of $i_{\bullet}$ and $p_{\bullet}$. 
The {\color{ultramarine}{\textbf{Gerstenhaber bracket}}} 
\begin{equation}
\label{eq:gerstenhaber-2}
    [\hskip 0.6mm , ] : \operatorname{H}^{m}\big(\Hom_{A^e}(P_{\bullet}(A),A)\big) \times \operatorname{H}^{n}\big(\Hom_{A^e}(P_{\bullet}(A),A)\big) \rightarrow \operatorname{H}^{m+n-1}\big(\Hom_{A^e}(P_{\bullet}(A),A)\big)
\end{equation}
for all $m, n \in \NN_{0}$ is then defined by transport of structures. 
More generally, given cocycles $f \in \Hom_{A^e}(P_m,A)$ and $g\in \Hom_{A^e}(P_n,A)$, we define the Gerstenhaber bracket $[f, g] \in \operatorname{H}^{m+n-1}(\Hom_{A^e}(P_{\bullet},A))\cong \HH^{m+n-1}(A)$ as the cohomology class of $i^{*}_{\bullet}([p^{*}_{\bullet}(f), p^{*}_{\bullet} (g)])$.

The following properties of the Gerstenhaber bracket are classical (see for instance \cite{MR161898}, equation (2), \textit{cf.} \cite{Sarah}, Lemmas 1.4.3 and 1.4.7).
\begin{lem}
\label{property of G}
Let $\Bbbk$ be a field and $A$ a $\Bbbk$-algebra.  
Then 
\begin{equation}
    \label{property 1}
\begin{split}
[x,y]=-(-1)^{(m-1)(n-1)}[y,x] \text{ and } [x, [y,z]] = [[x,y],z] + (-1)^{(m-1)(n-1)} [y,[x,z]],
\end{split}
\end{equation}
and
\begin{equation}
    \label{property 2}
    \begin{split}
    [x\cup y,z]= [x,z]\cup y+ (-1)^{m(p-1)} x\cup [y,z]
    \end{split}
\end{equation}
for all $x\in \HH^m(A)$, $y\in \HH^n(A)$ and $z\in \HH^p(A)$.
\end{lem}

The previous result is typically rephrased by stating that the Hochschild cohomology is a \textbf{\textcolor{ultramarine}{Gerstenhaber algebra}}, \textit{i.e.} a graded-commutative algebra $H = \oplus_{n \in \NN_{0}} H^{n}$ endowed with a bracket $[\hskip 0.6mm , ] : H \otimes H \rightarrow H$ satisfying $[H^{m},H^{n}] \subseteq H^{m+n-1}$ for $ m,n \in \NN_{0}$, \eqref{property 1} and \eqref{property 2}.

Assume for the rest of this subsection that $A$ is  \textbf{\textcolor{ultramarine}{graded}}, \textit{i.e.} there exist $\Bbbk$-vector subspaces $\{A_n\}_{n\in \ZZ }$ of $A$ such that $A=\oplus_{n\in \ZZ }A_n$, 
and $A_m\cdot A_n\subseteq A_{m+n}$ for all $m,n\in \ZZ$.
Recall that a left $A$-module $M$ is called \textbf{\textcolor{ultramarine}{graded}} if there are $\Bbbk$-vector subspaces $\{M_n\}_{n\in \ZZ }$ of $M$ such that $M=\oplus_{n\in \ZZ }M_n$ and $A_m\cdot M_n\subseteq M_{m+n}$ for all $m, n\in \ZZ $.
Given two graded left $A$-modules $M$ and $N$, a morphism $f:M\to N$ of left $A$-modules is called 
\textbf{\textcolor{ultramarine}{homogeneous}} of degree $d\in \ZZ$ if $f(M_n)\subseteq N_{n+d}$ for all $n\in \ZZ$.
Let $\Hom_A(M,N)$ be the $\Bbbk$-vector space consisting of all morphisms of left $A$-modules from $M$ to $N$.
Let 
\begin{equation}
\label{eq:Homm}
    \begin{split}
        \Homm_A(M,N)=\bigoplus_{d\in \ZZ} \Homm_A(M,N)_d,
    \end{split}
\end{equation}
be the graded $\Bbbk$-vector space, where $\Homm_A(M,N)_d$ is the subspace of $\Hom_A(M,N)$ consisting of all homogeneous morphisms of degree $d$. 


The following result is classical (see \cite{MR2046303}, Cor. 2.4.4). 
\begin{lem}
\label{lem:gr M N}
If $M$ is a finitely generated graded module over a graded algebra $A$, then $\Hom_A(M,N)=\Homm_A(M,N)$. 
\end{lem}


\begin{cor}
\label{cor:gr Hom}
Let $(P_{\bullet},\partial_{\bullet})$ 
\[     
\cdots \overset{\partial_{3}}{\longrightarrow} P_2
\overset{\partial_{2}}{\longrightarrow} P_1
\overset{\partial_{1}}{\longrightarrow} P_0 
\overset{\mu}{\longrightarrow}  A
\longrightarrow 0     
\] 
be a projective bimodule resolution of a graded $\Bbbk$-algebra $A$, where $P_i$ is finitely generated as left $A^e$-module for $i\in \NN_0$, and $\mu$ and $\partial_i$ are homogeneous of degree $0$ for $i\in\NN$.
Then $\Hom_{A^e}(P_i,A)=\Homm_{A^e}(P_i,A)$ for $i\in \NN_0$. 
Hence, 
the Hochschild cohomology $\HH^{\bullet}(A)\cong \oplus_{i\in\NN_0} \operatorname{H}^i(\Hom_{A^e}(P_{\bullet},A))$ of $A$ is a bigraded algebra, for the cohomological degree $i$ and the internal degree induced by that of $A$ and $P_{\bullet}$. 
Moreover, the cup product and the Gerstenhaber bracket on $\HH^{\bullet}(A)$ preserve the internal degree.
\end{cor}

\begin{rk}
The existence of  a projective bimodule resolution of the graded $\Bbbk$-algebra $A$ satisfying the conditions of the previous corollary clearly holds if the graded $\Bbbk$-algebra $A^{e}$ is noetherian (\textit{e.g.} if $A$ is finite dimensional over $\Bbbk$). 
\end{rk}

\subsection{\texorpdfstring{Method computing the bracket between $\HH^{0}(A)$ and $\HH^{n}(A)$}{Method computing the bracket between HH0(A) and HHn(A)}}
\label{subsection:Gerstenhaber brackets of HH^0} 

In this subsection we introduce an elementary method to compute the Gerstenhaber bracket between the cohomology groups $\HH^0(A)$ and $\HH^n(A)$ for $n \in \NN_{0}$ of any algebra $A$ using any projective bimodule resolution of $A$ (see Thm. \ref{HH0}). 
We were unable to explicitly find this method in the existing literature (see Remark \ref{rk:comparison to literature}), although we suspect it could be well known to the experts. 

Let $\rho$ be an element of the center $\Z(A)\cong \HH^0(A)$ of $A$ and $\ell_{\rho} \in \Hom_{A^e}(B_0(A),A)$ be the morphism defined by $\ell_{\rho}(1|1)=\rho$. 
Let $(P_{\bullet},\partial_{\bullet})$ be a projective bimodule resolution over $A$ with augmentation $\mu:P_0\to A$, and let 
$i_{0} : P_{0}\to B_{0}(A)$ be the $0$-th component of a morphism
$i_{\bullet} : P_{\bullet}\to B_{\bullet}(A)$ of complexes of $A$-bimodules lifting $\id_A$. 
The main aim of this subsection is to prove the following theorem, that tells us that we can compute the Gerstenhaber bracket between $\HH^{0}(A)$ and $
\HH^n(A)$ for $n\in \NN_0$ using a simple homological procedure on any projective bimodule resolution of $A$.
\begin{thm}
\label{HH0} 
Consider the same assumptions as in the previous paragraph. 
Let $\eta_n:P_n\to P_n$ be the map given by $\eta_n(v)=\rho v - v \rho$ for $v\in P_n$ and $n\in \NN_0$. 
Since $\eta_{\bullet}=\{ \eta_n :P_n\to P_n\}_{n\in\NN_0}$ is a lifting of the zero morphism from $A$ to itself, $\eta_{\bullet}$ is null-homotopic, \textit{i.e.} there is a family $h^{\rho}_{\bullet}=\{ h^{\rho}_n :P_n\to P_{n+1}\}_{n\in\NN_0}$ of morphisms of $A$-bimodules such that 
\begin{equation}
    \label{eq:eta}
\begin{split}
    \eta_0=\partial_1h^{\rho}_0
    \text{ $\phantom{x}$ and $\phantom{x}$ } 
    \eta_n=h^{\rho}_{n-1}\partial_n+\partial_{n+1}h^{\rho}_n
\end{split}
\end{equation}
for $n\in\NN$.
Then, if $\phi \in \Hom_{A^e}(P_n,A)$ is a cocycle for some $n\in\NN_0$, the Gerstenhaber bracket $[\phi, \ell_{\rho} i_{0}] \in \HH^{n-1}(A)$ is given by the cohomology class of $\phi h^{\rho}_{n-1}$. 
\end{thm}

\begin{rk}
\label{rk: projective cocycle boundary}
It is easy to see that if $\phi \in \Hom_{A^e}(P_n,A)$ is a cocycle (resp., coboundary), then $\phi h^{\rho}_{n-1}$ is a cocycle (resp., coboundary) by applying \eqref{eq:eta}. 
On the other hand, in general we have $P_{0} = B_{0} = A \otimes A$ and $i_{0} = \operatorname{id}_{A \otimes A}$, so we can forget about $i_{0}$ in Theorem \ref{HH0}.  
\end{rk}

The rest of this subsection is devoted to proving Theorem \ref{HH0}. 
In order to do that, we first need to prove some preliminary results. 

Let $t_n:B_{n}(A)\to B_{n+1}(A)$ be the morphism of $A$-bimodules given by 
\begin{equation}\label{t}
    \begin{split}
t_n(a_0|\dots|a_{n+1})=\sum_{j=0}^n(-1)^j a_0|\dots|a_j|\rho|a_{j+1}|\dots|a_{n+1}
    \end{split}
\end{equation}
for $a_0,\dots,a_{n+1}\in A$ and $n\in\NN_0$.
Let $\xi_{\bullet}=\{ \xi_n :B_n(A)\to B_n(A)\}_{n\in\NN_0}$ be the family of morphisms of $A$-bimodules defined by $\xi_n(u)=\rho u-u \rho $ for $u\in B_n(A)$ and $n\in\NN_0$. 
\begin{lem}
    \label{lem:bar1}
    It is easy to check that $\xi_0=d_1t_0$ and $\xi_n=t_{n-1}d_n+d_{n+1}t_n$ for $n\in \NN$.
\end{lem}
\begin{proof}
For $a_0,\dots,a_{n+1}\in A$ and $n\in \NN$, 
\[ 
d_1 t_0(a_0|a_1)=d_1(a_0|\rho|a_1)=a_0\rho|a_1-a_0|\rho a_1=\rho a_0|a_1-a_0|a_1\rho=\xi_0(a_0|a_1),
\] 
and 
\[ 
    d_{n+1}t_n(a_0|\dots|a_{n+1})
    =d_{n+1}\bigg(\sum_{j=0}^n(-1)^j a_0|\dots|a_j|\rho|a_{j+1}|\dots|a_{n+1}\bigg)
    =S_1+S_2,
\] 
where
\begin{align*}
S_1 
&= \sum_{j=0}^n(-1)^j \bigg\{
    (-1)^j a_0|\dots|a_{j-1}|a_j\rho|a_{j+1}|\dots|a_{n+1} 
    +(-1)^{j+1}a_0|\dots|a_j|\rho a_{j+1}|a_{j+2}|\dots|a_{n+1}
    \bigg\}
\\
&= a_0\rho|\dots|a_{n+1}-a_0|\dots|a_n|\rho a_{n+1}
\\
&=\xi_n(a_0|\dots|a_{n+1}), 
\end{align*}
and 
\begin{align*}
S_2 
&=\sum_{j=0}^n(-1)^j \bigg\{   
    \sum_{r=0}^{j-2}(-1)^r a_0|\dots|a_{r-1}|a_ra_{r+1}|a_{r+2}|\dots |a_j|\rho|a_{j+1}|\dots|a_{n+1}
\\ 
& \phantom{= \;}
+(-1)^{j-1}a_0|\dots|a_{j-2}|a_{j-1}a_j|\rho|a_{j+1}|\dots|a_{n+1}
\\  
& \phantom{= \;} 
+(-1)^{j+2}a_0|\dots|a_j|\rho|a_{j+1}a_{j+2}|a_{j+3}|\dots|a_{n+1}
\\  
& \phantom{= \;}
+\sum_{r=j+2}^{n}(-1)^{r+1}a_0|\dots|a_j|\rho|a_{j+1}|\dots|a_{r-1}|a_ra_{r+1}|a_{r+2}|\dots|a_{n+1}
\bigg\}
\\
&= \sum_{i=0}^n \bigg\{  
  \sum_{j=i+2}^n (-1)^j (-1)^i a_0|\dots|a_{i-1}|a_i a_{i+1}|a_{i+2}|\dots|a_j|\rho|a_{j+1}|\dots|a_{n+1} 
\\
& \phantom{= \;} 
-a_0|\dots|a_{i-1}|a_i a_{i+1}|\rho|a_{i+2}|\dots|a_{n+1}  
+a_0|\dots|a_{i-1}|\rho|a_ia_{i+1}|a_{i+2}|\dots|a_{n+1}
\\
& \phantom{= \;} 
+\sum_{j=0}^{i-2}(-1)^j(-1)^{i+1}a_0|\dots |a_j |\rho|a_{j+1}|\dots|a_{i-1}|a_i a_{i+1}|a_{i+2}|\dots|a_{n+1}
\bigg\}
\\
&= -\sum_{i=0}^n (-1)^i\bigg\{ 
    \sum_{j=0}^{i-2}(-1)^j a_0|\dots|a_j|\rho|a_{j+1}|\dots|a_{i-1}|a_ia_{i+1}|a_{i+2}|\dots|a_{n+1}
\\
& \phantom{= \;}   
+(-1)^{i-1}a_0|\dots|a_{i-1}|\rho|a_ia_{i+1}|a_{i+2}|\dots|a_{n+1}
+(-1)^i a_0|\dots |a_{i-1}|a_i a_{i+1}|\rho|a_{i+2}|\dots|a_{n+1}
\\
& \phantom{= \;}   
+\sum_{j=i+2}^n (-1)^{j-1}a_0|\dots|a_{i-1}|a_ia_{i+1}|a_{i+2}|\dots|a_j|\rho|a_{j+1}|\dots|a_{n+1}
\bigg\}
\\
&= - t_{n-1}d_n(a_0|\dots|a_{n+1}).
\end{align*}
Hence, $\xi_n=t_{n-1}d_n+d_{n+1}t_n$.
\end{proof}

\begin{lem}
\label{lem:bar-G-bracket}
The Gerstenhaber bracket $[\varphi, \ell_{\rho}]\in \Hom_{A^e}(B_{n-1}(A),A)$ is given by $[\varphi, \ell_{\rho}]=\varphi t_{n-1}$ for $\varphi\in \Hom_{A^e}(B_n(A),A)$ and $n\in\NN_0$.
\end{lem}
\begin{proof}
For $a_0, \dots,a_{n}\in A$, 
\begin{align*}
[\varphi, \ell_{\rho}](a_0|\dots |a_{n})&=a_0[F(\varphi),F(\ell_{\rho})](a_1|\dots|a_{n-1})a_n
\\ 
&= a_0 (F(\varphi)\circ_{G} F(\ell_{\rho}))(a_1|\dots|a_{n-1})  a_n
\\ 
&=a_0\bigg( \sum_{i=1}^{n} (-1)^{i-1}F(\varphi)(a_1|\dots|a_{i-1}|\rho|a_i|\dots|a_{n-1})  \bigg)a_n
\\  
&= a_0\bigg( \sum_{i=1}^{n} (-1)^{i-1}\varphi(1|a_1|\dots|a_{i-1}|\rho|a_i|\dots|a_{n-1}|1)  \bigg)a_n
\\  
&= \sum_{i=1}^{n} (-1)^{i-1}\varphi(a_0|a_1|\dots|a_{i-1}|\rho|a_i|\dots|a_{n-1}|a_n)  
\\  
&= (\varphi t_{n-1})(a_0|a_1|\dots|a_{n-1}|a_n).
\end{align*}
Hence, $[\varphi, \ell_{\rho}]=\varphi t_{n-1}$.
\end{proof}


\begin{lem}
\label{lem: exist s}
We assume the same hypotheses as those of Theorem \ref{HH0}. 
Then, there exists a family $s_{\bullet}=\{ s_n :P_n\to B_{n+2}(A)\}_{n\in\NN_0}$ of morphisms of $A$-bimodules 
such that 
$i_{1}h^{\rho}_0-t_0i_0=d_2s_0$ and $i_{n+1}h^{\rho}_n-t_ni_n=d_{n+2}s_n-s_{n-1}\partial_n$ for $n\in\NN$.
\end{lem}
\begin{proof}
Since $d_1(i_1h^{\rho}_0-t_0i_0)=i_0\partial_1 h^{\rho}_0 -d_1t_0i_0=i_0\eta_0-\xi_0 i_0=0$, where we used that $i_{0}$ is a morphism of $A$-bimodules in the last equality, there exists a morphism $s_0:P_0\to B_2(A)$ of $A$-bimodules such that $d_2s_0=i_1h^{\rho}_0-t_0i_0$. 
We now claim that there exists a family $s_{\bullet}=\{ s_n :P_n\to B_{n+2}(A)\}_{n\in\NN_0}$
of morphisms of $A$-bimodules such that 
$d_{n+2}s_n=i_{n+1}h^{\rho}_n-t_ni_n+s_{n-1}\partial_n$ by induction on $n\in\NN_0$ (where $s_{-1} = 0$). 
Indeed, 
\begin{align*}
d_{n+1}(i_{n+1}h^{\rho}_n-t_ni_n+s_{n-1}\partial_n)
&=d_{n+1}i_{n+1}h^{\rho}_n-d_{n+1}t_ni_n+(i_{n}h^{\rho}_{n-1}-t_{n-1}i_{n-1}+s_{n-2}\partial_{n-1})\partial_n
\\  
&=i_n\partial_{n+1}h^{\rho}_n-(\xi_n-t_{n-1}d_n)i_n+i_nh^{\rho}_{n-1}\partial_n-t_{n-1}i_{n-1}\partial_n
\\   
&=i_n(\partial_{n+1}h^{\rho}_n+h^{\rho}_{n-1}\partial_n)-\xi_n i_n
=i_n \eta_n-\xi_n i_n =0,
\end{align*}
where we used the inductive assumption in the first equality, Lemma \ref{lem:bar1} in the second equality, the definition of $\eta_{n}$ in the third equality and the fact that $i_{n}$ is a morphism of $A$-bimodules in the last equality. 
The result thus follows.
\end{proof}

\begin{proof}[Proof of Theorem \ref{HH0}]
Let $\varphi=\phi p_n\in \Hom_{A^e}(B_{n}(A),A)$. 
Then $\varphi$ is a cocycle and 
$[\phi, i^{*}_{\bullet}(\ell_{\rho})]=i^{*}_{\bullet}([p^{*}_{\bullet}(\phi), \ell_{\rho}])=[\varphi, \ell_{\rho}]i_{n-1}=\varphi t_{n-1}i_{n-1}$ by Lemma \ref{lem:bar-G-bracket}. 
Since $p_{\bullet} i_{\bullet}$ is homotopic to the identity of $P_{\bullet}$, there exists $\phi_1\in \Hom_{A^e}(P_{n-1}, A)$ such that $\phi-\varphi i_n =\phi-\phi p_n i_n=\phi_1 \partial_n$. 
Then, 
\begin{align*}
\phi h^{\rho}_{n-1}-\varphi t_{n-1}i_{n-1}
 &=
(\varphi i_n+\phi_1 \partial_n) h^{\rho}_{n-1}-\varphi t_{n-1}i_{n-1} 
\\
&= \varphi(d_{n+1}s_{n-1}-s_{n-2}\partial_{n-1})+\phi_1(\eta_{n-1}-h^{\rho}_{n-2}\partial_{n-1})
\\ 
&= -\varphi s_{n-2}\partial_{n-1}-\phi_1 h^{\rho}_{n-2}\partial_{n-1}\in \Hom_{A^e}(P_{n-1}, A)
\end{align*}
is a boundary, where we used Lemma \ref{lem: exist s} and the definition of $\eta_{n-1}$ in the second equality and the fact that $\phi_{1}$ is a morphism of $A$-bimodules in the last identity.
Hence, $[\phi, i^{*}_{\bullet}(\ell_{\rho})] \in \HH^{n-1}(A)$ coincides with the cohomology class of $\phi h^{\rho}_{n-1}$, as was to be shown.
\end{proof}

\begin{rk}
\label{rk:comparison to literature}
The homotopy maps $h^{\rho}_{\bullet}$ in Theorem \ref{HH0} 
are presumably homotopy liftings in the sense of \cite{MR3974969}. 
However, our maps $h^{\rho}_{\bullet}$ do not directly follow the scheme of that definition --as well as being far simpler, for they are restricted to a much easier situation-- since they do not require the computation of any map $\Delta : P_{\bullet} \rightarrow P_{\bullet} \otimes_{A} P_{\bullet}$ lifting the isomorphism $A \rightarrow A \otimes_{A} A$, which is also the case in \cite{MR3498646}.
\end{rk}

\subsection{\texorpdfstring{Method computing the bracket between $\HH^{1}(A)$ and $\HH^{n}(A)$ (after M. Suárez-Álvarez)}{Method computing the bracket between HH1(A) and HHn(A) (after M. Suárez-Álvarez)}}
\label{subsection:Gerstenhaber brackets of HH^1} 

In this subsection we will briefly recall the method introduced by M. Suárez-Álvarez in \cite{Mariano} 
to compute the Gerstenhaber bracket between $\HH^1(A)$  and $\HH^n(A)$ for $n \in \NN_{0}$. 

Recall that $\HH^1(A)$ is isomorphic to the quotient of the space of derivations of $A$ modulo the subspace of inner derivations.
Let $\rho: A\to A$ be a derivation of $A$, \textit{i.e.} $\rho(xy)=\rho(x)y+x\rho(y)$ for all $x,y\in A$. 
For a left $A$-module $M$, a {\color{ultramarine}{\textbf{$\rho$-operator}}} on $M$ is a map $f:M\to M$ such that $f(am)=\rho (a)m+af(m)$ for all $a\in A$ and $m\in M$.
It is direct to see that the map $\rho^e=\rho\otimes \id_A+\id_A\otimes \rho:A^e\to A^e$ defined by $\rho^e(x\otimes y)=\rho(x)\otimes y+x\otimes \rho(y)$ for $x,y\in A$ is a derivation of the enveloping algebra $A^e$ and $\rho$ is a $\rho^e$-operator on $A$. 

Let $(P_{\bullet},\partial_{\bullet})$ be a projective bimodule resolution over $A$ with augmentation $\mu : P_{0} \rightarrow A$. 
A {\color{ultramarine}{\textbf{$\rho^e$-lifting}}} of $\rho$ to $(P_{\bullet},\partial_{\bullet})$ is a family of $\rho^e$-operators $\rho_{\bullet}=\{ \rho_n: P_n\to P_n\}_{n\in\NN_0}$ such that $\mu \rho_0=\rho \mu$ and $\partial_n \rho_n=\rho_{n-1} \partial_n$ for $n\in\NN$.
The morphism of complexes 
\[     \rho_{\bullet,P_{\bullet}}^{\sharp}:\Hom_{A^e}(P_{\bullet},A)\to \Hom_{A^e}(P_{\bullet},A)     \] 
defined by $\rho_{n,P_{\bullet}}^{\sharp}(\phi)=\rho \phi-\phi \rho_n$ for $\phi\in \Hom_{A^e}(P_{n},A)$ and $n\in\NN_0$ is independent of the $\rho^e$-lifting up to homotopy (see \cite{Mariano}, Lemma 1.6) and it thus induces a morphism on cohomology that we will denote by the same symbol. 
Let $i_{\bullet}:P_{\bullet}\to B_{\bullet}(A)$ and $p_{\bullet}:B_{\bullet}(A)\to P_{\bullet}$ be morphisms of complexes of $A$-bimodules lifting $\id_A$.
Then the diagram
\begin{equation}
\label{eq:com b p}
\begin{tikzcd}[column sep= 40pt]
\operatorname{H}^{n}\big(\Hom_{A^e}(B_{\bullet}(A),A)\big) \ar[d,"\operatorname{H}(i^{*}_{\bullet})"] 
\ar[r, "\operatorname{H}(\rho_{\bullet,B_{\bullet}(A)}^{\sharp})"] 
& \operatorname{H}^{n}\big(\Hom_{A^e}(B_{\bullet}(A),A)\big) \ar[d,"\operatorname{H}(i^{*}_{\bullet})"]  
\\
\operatorname{H}^{n}\big(\Hom_{A^e}(P_{\bullet},A)\big)
\ar[r, "\operatorname{H}(\rho_{\bullet,P_{\bullet}}^{\sharp})"]
&{\operatorname{H}^{n}\big(\Hom_{A^e}(P_{\bullet},A)\big)}
\end{tikzcd}
\end{equation} 
commutes (see \cite{Mariano}, Lemma 1.6). 
On the other hand, as noted in \cite{Mariano}, Sections 2.1 and 2.2, using the $\rho^e$-lifting of $\rho$ to the bar resolution defined by 
\begin{equation}
    \begin{split}
        \rho_n(a_0|\dots|a_{n+1})=\sum_{j=0}^{n+1} a_0|\dots|a_{j-1}|\rho(a_j)|a_{j+1}|\dots|a_{n+1}
    \end{split}
\end{equation}
for $a_0,\dots,a_{n+1}\in A$ and $n\in\NN_0$, 
it is easy to check that the diagram 
\begin{equation}
\label{eq:com b c}
\begin{tikzcd}[column sep= 40pt]
\Hom_{A^e}\big(B_{n}(A),A\big) \ar[d, "F"] 
\ar[r, "\rho_{n,B_{\bullet}(A)}^{\sharp}"] 
& 
\Hom_{A^e}\big(B_{n}(A),A\big) \ar[d, "F"]  
\\
\Hom_{\Bbbk}(A^{\otimes n},A)
\ar[r, "{[\rho , - ]}"]
&
\Hom_{\Bbbk}(A^{\otimes n},A)
\end{tikzcd}
\end{equation} 
commutes. 
As a consequence, the Gerstenhaber bracket between the cohomology classes of $G(\rho) \in \Hom_{A^e}(B_{1}(A),A)$ and $\varphi \in \Hom_{A^e}(B_{n}(A),A)$ is given by the cohomology class of $[G(\rho), \varphi]=G([\rho,F(\varphi)])=\rho_{n,B_{\bullet}(A)}^{\sharp}(\varphi)$. 

We finally recall one of the main results of \cite{Mariano}, which tells us that we can compute the Gerstenhaber bracket between $\HH^{1}(A)$ and $
\HH^n(A)$ for $n\in \NN_0$ using any projective bimodule resolution of $A$ (see \cite{Mariano}, Thm. A and Section 2.2).
The proof just follows from observing that, on cohomology, \eqref{eq:com b p} gives us the identities
\[     \big[i^{*}_{\bullet}(G(\rho)), \phi\big]
=i^{*}_{\bullet}\big([G(\rho),p^{*}_{\bullet}(\phi)]\big)
=i^{*}_{\bullet} \big(\rho_{n,B_{\bullet}(A)}^{\sharp}(p^{*}_{\bullet}( \phi))\big)
=\rho_{n,P_{\bullet}}^{\sharp}(\phi).     \]

\begin{thm}
\label{thm:HH1 HHn} 
Let $(P_{\bullet},\partial_{\bullet})$ be a projective bimodule resolution over the algebra $A$ with augmentation $\mu : P_{0} \rightarrow A$, and let 
$i_{1} : P_{1}\to B_{1}(A)$ be the first component of the morphism $i_{\bullet} : P_{\bullet}\to B_{\bullet}(A)$ of complexes of $A$-bimodules lifting $\id_A$.  
Given a cocycle $\phi\in \Hom_{A^e}(P_{n},A)$ and $n\in \NN_0$, the Gerstenhaber bracket $[G(\rho) i_{1}, \phi] \in \HH^n(A)$ is given by the cohomology class of $\rho_{n,P_{\bullet}}^{\sharp}(\phi)$.
\end{thm}

\begin{rk}
Note that in our Theorem \ref{HH0}, as well as in the result proved in \cite{Mariano} that was recalled before as Theorem \ref{thm:HH1 HHn}, we need at least some component(s) of the comparison map from the generic projective resolution $(P_{\bullet},\partial_{\bullet})$ to the bar resolution. 
\end{rk}


\section{\texorpdfstring{Basics on the Fomin-Kirillov algebra $\FK(3)$ on $3$ generators}{Basics on the Fomin-Kirillov algebra FK(3) on 3 generators}}
\label{sec:basics}

In this section we will review the basic definitions concerning the Fomin-Kirillov algebra on $3$ generators.
Recall that, 
given $i\in\ZZ$, we will denote by $\ZZ_{\leqslant i}$ the set $\{m\in\ZZ|m\leqslant i \}$.
Given $r \in \RR$, we set $\lfloor r\rfloor =\sup \{n\in \ZZ|n\leqslant r \}$ the usual \textcolor{ultramarine}{\textbf{floor function}}. 
From now on, $\Bbbk$ is a field of characteristic different from $2$ and $3$.

We recall that the 
\textcolor{ultramarine}{\textbf{Fomin–Kirillov algebra}}
on $3$ generators is the $\Bbbk$-algebra $\FK(3)$ generated by the $\Bbbk$-vector space $V$ spanned by three elements $a$, $b$ and $c$ modulo the ideal generated by the vector space $R \subseteq V^{\otimes 2}$ spanned by 
\begin{equation}
    \{ a^2, b^2, c^2, ab+bc+ca, ba+ac+cb \}. 
\end{equation}   
This is a nonnegatively graded algebra by setting the generators $a$, $b$ and $c$ in (internal) degree $1$.
As usual, we will omit the tensor symbol $\otimes$ when denoting the product of the elements of the tensor algebra $\T V= \oplus_{n \in \NN_{0}} V^{\otimes n} $. 
We refer the reader to \cites{FK99,MS00} for more information on Fomin-Kirillov algebras.
Note that $\FK(3)=\oplus_{m\in \llbracket 0,4 \rrbracket} \FK(3)_m$, where $\FK(3)_m$ is the subspace of $\FK(3)$ concentrated in internal degree $m$. 
It is easy to see that 
\begin{equation}
\label{eq:basis-FK}
\mathcalboondox{B}=\{1,a,b,c,ab,bc,ba,ac,aba,abc,bac,abac  \}   
\end{equation}     
is a basis of $\FK(3)$ (see \cite{FK99}). 
Given $m\in \llbracket 0,4 \rrbracket$, we will denote by $\mathcalboondox{B}_m$ the subset of \eqref{eq:basis-FK} that is a basis of $\FK(3)_{m}$. 

Let us briefly denote by $\{A,B,C\}$ the basis of $V^*$ dual to the basis $\{a,b,c\}$ of $V$, where the former are concentrated in internal degree $-1$.
The \textcolor{ultramarine}{\textbf{quadratic dual}} $\FK(3)^!=\oplus_{n\in\NN_0} \FK(3)^!_{-n}$ of the Fomin Kirillov algebra $\FK(3)$ is then given by 
\[     \FK(3)^!=\Bbbk\langle A,B,C\rangle/ (BA-AC, CA-AB, AB-BC, CB-BA),     \]
where $\FK(3)^!_{-n}$ is the subspace of $\FK(3)^! $ concentrated in internal degree $-n$.
Notice that $\FK(3)^!_0=\Bbbk$ and $\FK(3)^!_{-1}\cong V^*$. 
We recall that $\mathcalboondox{B}^!_n=\{A^n, B^n, C^n, A^{n-1}B, A^{n-1}C, A^{n-2}B^2 \}$ is a basis of $\FK(3)^!_{-n}$ for all integers $n \geqslant 2$, where we follow the convention that $A^{0}B^{2} = B^{2}$ (see \cite{SV16}, Lemma 4.4). 

For simplicity, from now on we will denote the Fomin-Kirillov algebra $\FK(3)$ simply by $A$.
Let $(A_{-n}^!)^*$ be the dual space of $A_{-n}^!$ and $\mathcalboondox{B}^{!*}_n = \{\alpha_n, \beta_n,\gamma_n,\alpha_{n-1}\beta, \alpha_{n-1}\gamma, \alpha_{n-2}\beta_2\} \setminus \{ \mathbf{0} \}$ the dual basis to $\mathcalboondox{B}^!_n$ for $n\in \NN$, where we will follow the convention that if the index of some letter in an element of the previous sets is less than or equal to zero, it is the zero element $\mathbf{0}$.
We will omit the index $1$ for the elements of the previous bases and write $\mathcalboondox{B}^{!*}_0 = \{\epsilon^!\}$, where $\epsilon^!$ is the basis element of $(A_{0}^!)^*$. 
The previous bases for the homogeneous components of $A$ or $(A^!)^{\#}  =\oplus_{n\in\NN_0}(A_{-n}^!)^*$ will be called {\color{ultramarine}{\textbf{usual}}}. 
Recall that  $(A^!)^{\#}$ is a graded bimodule over $A^!$ via $(ufv)(w)=f(vwu)$ for $u,v,w \in A^!$ and $f\in (A^!)^{\#}$. 

For the explicit description of the {\color{ultramarine}{\textbf{bimodule Koszul complex}}} $(K^b_{\bullet}, d^b_{\bullet})$ over the Fomin-Kirillov algebra $A$, we refer the reader to \cite{es zl}, Subsection 3.1.
The following result gives an explicit description of the minimal projective resolution of $A$ in the category of bounded below graded $A$-bimodules.

\begin{prop}
\label{pro:bimodule projective resolution}
(\cite{es zl}, Prop. 3.7)
The minimal projective resolution $(P^b_{\bullet},\delta^b_{\bullet})$ of $A$ in the category of bounded below graded $A$-bimodules is given as follows. For $n\in\NN_0$, set 
\[ P^b_n=\underset{i\in \llbracket 0, \lfloor n/4\rfloor \rrbracket}{\bigoplus}\omega_i K^b_{n-4i}
=\underset{i\in \llbracket 0, \lfloor n/4\rfloor \rrbracket}{\bigoplus}\omega_i A\otimes (A^!_{-(n-4i)})^*\otimes A,\]
where $\omega_i$ is a symbol of homological degree $4i$ and internal degree $6i$ for all $i\in\NN_0$, the $A$-bimodule structure of $P^b_n$ is given by 
$x'(\omega_i x\otimes u\otimes y)y'=\omega_i x'x\otimes u\otimes yy'$ for all $x,x',y,y'\in A$ and $u\in (A^!_{-(n-4i)})^*$, and the differential $\delta^b_{n}:P^b_n\to P^b_{n-1}$ for $n\in \NN$ is given by 
\[\delta^b_{n}\bigg( \underset{i\in \llbracket 0, \lfloor n/4\rfloor \rrbracket}{\sum} \omega_i\rho_{n-4i}\bigg)=\underset{i\in \llbracket 0, \lfloor n/4\rfloor \rrbracket}{\sum} \big( \omega_i d^b_{n-4i}(\rho_{n-4i})+\omega_{i-1}f^b_{n-4i}(\rho_{n-4i}) \big),\]
where $\rho_j\in K^b_j$ for $j\in\NN_0$, $\omega_{-1}=0$ and $f^b_j:K^b_j\to K^b_{j+3}$ are the morphisms given in \cite{es zl}, (3.2). 
This gives a minimal projective resolution of $A$ by means of the augmentation $\epsilon^b:P^b_0=A\otimes (A^!_0)^*\otimes A\to A$, where $\epsilon^b(x\otimes \epsilon^!\otimes y)=xy$ for $x,y\in A$.
\end{prop}

To reduce space, we will typically use vertical bars instead of the tensor product symbols $\otimes$ for the elements of $P_{\bullet}$.
We will use in the sequel the following explicit expression of some values of the map
$f^b_0$ in the previous result. 
They follow immediately from \cites{es zl}, (3.2). 
\begin{fact}
\label{fact:fb0degree1}
The morphism $f^b_0$ defined in \cites{es zl}, (3.2), satisfies that 
\begin{align*}
    f^b_0(a|\epsilon^!|1)&=2a|\alpha_3|bac+2a|\beta_3|abc-2a|\gamma_3|aba-a|\alpha_2\beta|abc+a|\alpha_2\gamma|aba-a|\alpha\beta_2|bac
    -ac|\alpha_2\beta|ab
    \\
    & \phantom{= \;} 
    -ab|\alpha_2\gamma|ac+ab|\alpha\beta_2|ab+ab|\alpha\beta_2|bc-ac|\alpha\beta_2|ba-2ab|\alpha_3|ab-2ab|\alpha_3|bc
    \\
    & \phantom{= \;} 
    +2ac|\alpha_3|ba
    +2ac|\beta_3|ab+2ab|\gamma_3|ac-abc|\alpha_2\beta|a-aba|\alpha_2\beta|c+abc|\alpha_2\gamma|b
    \\
    & \phantom{= \;} 
    +aba|\alpha_2\gamma|a +2abc|\beta_3|a
    +2aba|\beta_3|c-2abc|\gamma_3|b-2aba|\gamma_3|a+2abac|\alpha_3|1
    \\
    & \phantom{= \;} 
    -abac|\alpha\beta_2|1,
    \\
    f^b_0(1|\epsilon^!|a)&=-2|\alpha_3|abac+1|\alpha\beta_2|abac-a|\alpha_2\beta|abc-c|\alpha_2\beta|aba+a|\alpha_2\gamma|aba+b|\alpha_2\gamma|abc
    \\
    & \phantom{= \;} 
    +2a|\beta_3|abc+2c|\beta_3|aba-2a|\gamma_3|aba-2b|\gamma_3|abc+ba|\alpha_2\beta|(ab+bc)+ab|\alpha_2\gamma|ba
    \\
    & \phantom{= \;} 
    +bc|\alpha_2\gamma|ba+ab|\alpha\beta_2|(ab+bc)-ac|\alpha\beta_2|ba-2ab|\alpha_3|(ab+bc)+2ac|\alpha_3|ba
    \\
    & \phantom{= \;} 
    -2ba|\beta_3|(ab+bc)-2ab|\gamma_3|ba-2bc|\gamma_3|ba+2bac|\alpha_3|a+2abc|\beta_3|a-2aba|\gamma_3|a
    \\
    & \phantom{= \;} 
    -abc|\alpha_2\beta|a+aba|\alpha_2\gamma|a-bac|\alpha\beta_2|a,
    \\
    f^b_0(b|\epsilon^!|1)&=2b|\alpha_3|bac+2b|\beta_3|abc-2b|\gamma_3|aba-b|\alpha_2\beta|abc+b|\alpha_2\gamma|aba-b|\alpha\beta_2|bac+ba|\alpha_2\beta|ac
    \\
    & \phantom{= \;} 
    +ba|\alpha_2\beta|ba-bc|\alpha_2\beta|ab-ba|\alpha_2\gamma|bc-bc|\alpha\beta_2|ba+2bc|\alpha_3|ba-2ba|\beta_3|ac
    \\
    & \phantom{= \;} 
    -2ba|\beta_3|ba 
    +2bc|\beta_3|ab+2ba|\gamma_3|bc
    +aba|\alpha_2\gamma|b+bac|\alpha_2\gamma|a-aba|\alpha\beta_2|c
    \\
    & \phantom{= \;} 
    -bac|\alpha\beta_2|b
    +2aba|\alpha_3|c+2bac|\alpha_3|b-2aba|\gamma_3|b-2bac|\gamma_3|a+2abac|\beta_3|1
    \\
    & \phantom{= \;} 
    -abac|\alpha_2\beta|1,
    \\
    f^b_0(1|\epsilon^!|b)&=-2|\beta_3|abac+1|\alpha_2\beta|abac+a|\alpha_2\gamma|bac+b|\alpha_2\gamma|aba-b|\alpha\beta_2|bac-c|\alpha\beta_2|aba
    \\
    & \phantom{= \;} 
    +2b|\alpha_3|bac+2c|\alpha_3|aba-2a|\gamma_3|bac-2b|\gamma_3|aba-bc|\alpha_2\beta|ab+ba|\alpha_2\beta|(ba+ac)
    \\
    & \phantom{= \;} 
    +ba|\alpha_2\gamma|ab+ac|\alpha_2\gamma|ab+ab|\alpha\beta_2|(ba+ac)-2ab|\alpha_3|(ba+ac)+2bc|\beta_3|ab
    \\
    & \phantom{= \;} 
    -2ba|\beta_3|(ba+ac)-2ba|\gamma_3|ab-2ac|\gamma_3|ab
    +2bac|\alpha_3|b+2abc|\beta_3|b-2aba|\gamma_3|b
     \\
    & \phantom{= \;} 
    -abc|\alpha_2\beta|b+aba|\alpha_2\gamma|b-bac|\alpha\beta_2|b,
    \\
    f^b_0(c|\epsilon^!|1)&=2c|\alpha_3|bac+2c|\beta_3|abc-2c|\gamma_3|aba-c|\alpha_2\beta|abc+c|\alpha_2\gamma|aba-c|\alpha\beta_2|bac
     \\
    & \phantom{= \;} 
    -(ab+bc)|\alpha_2\beta|ac-(ab+bc)|\alpha_2\beta|ba+(ab+bc)|\alpha_2\gamma|bc+(ba+ac)|\alpha_2\gamma|ac
    \\
    & \phantom{= \;} 
    -(ba+ac)|\alpha\beta_2|ab-(ba+ac)|\alpha\beta_2|bc+2(ba+ac)|\alpha_3|ab+2(ba+ac)|\alpha_3|bc
    \\
    & \phantom{= \;} 
    +2(ab+bc)|\beta_3|ac+2(ab+bc)|\beta_3|ba-2(ab+bc)|\gamma_3|bc-2(ba+ac)|\gamma_3|ac
     \\
    & \phantom{= \;} 
    +bac|\alpha_2\beta|a-abc|\alpha_2\beta|c-bac|\alpha\beta_2|c+abc|\alpha\beta_2|b+2bac|\alpha_3|c-2abc|\alpha_3|b
    \\
    & \phantom{= \;} 
    -2bac|\beta_3|a
    +2abc|\beta_3|c+2abac|\gamma_3|1-abac|\alpha_2\gamma|1,
    \\
    f^b_0(1|\epsilon^!|c)&=-2|\gamma_3|abac+1|\alpha_2\gamma|abac+a|\alpha_2\beta|bac-c|\alpha_2\beta|abc+b|\alpha\beta_2|abc-c|\alpha\beta_2|bac
    \\
    & \phantom{= \;} 
    -2b|\alpha_3|abc
    +2c|\alpha_3|bac-2a|\beta_3|bac+2c|\beta_3|abc-bc|\alpha_2\beta|ac+ab|\alpha_2\gamma|bc
    +bc|\alpha_2\gamma|bc
    \\
    & \phantom{= \;} 
    +ba|\alpha_2\gamma|ac
    +ac|\alpha_2\gamma|ac-ac|\alpha\beta_2|bc
    +2ac|\alpha_3|bc+2bc|\beta_3|ac-2ab|\gamma_3|bc-2bc|\gamma_3|bc
    \\
    & \phantom{= \;} 
    -2ba|\gamma_3|ac-2ac|\gamma_3|ac
    +2bac|\alpha_3|c+2abc|\beta_3|c-2aba|\gamma_3|c-abc|\alpha_2\beta|c+aba|\alpha_2\gamma|c
    \\
    & \phantom{= \;} 
    -bac|\alpha\beta_2|c.
    \end{align*}
\end{fact}

Given $n\in\NN_0$, there is a canonical isomorphism 
\begin{equation}
\label{X Y}
    \begin{split}
        \Hom_{A^e}(P_{n}^b,A)\cong Q^n 
    \end{split}
\end{equation}
of graded $\Bbbk$-vector spaces, 
where $Q^n= \oplus_ {i\in \llbracket 0, \lfloor n/4 \rfloor \rrbracket} \omega^*_i K^{n-4i}$ and $K^n=\operatorname{Hom}_{\Bbbk}((A_{-n}^!)^*,A)$. 
Transporting the differential of the left member of \eqref{X Y} 
induced by that of $(P^b_{\bullet},\delta^b_{\bullet})$, 
we obtain a complex of $\Bbbk$-vector spaces $Q^{\bullet}$, 
whose cohomology gives the linear structure of the Hochschild cohomology $\HH^{\bullet}(A)$. 
Note that the space $\operatorname{Hom}_{\Bbbk}((A_{-n}^!)^*,A_m)$ is concentrated in cohomological degree $n$ and internal degree $m-n$. 
The symbol $\omega^*_i$ has cohomological degree $4i$ and internal degree $-6i$ for $i\in\NN_0$. 
We will usually omit $\omega^*_0$ for simplicity. 

Let $H^n_m$ be the subspace of $\operatorname{H}^n(Q^{\bullet})$ concentrated in internal degree $m-n$ for $m,n\in \ZZ$. 
Note that $H^n_m=0$ for $(n,m)\in \ZZ^2\setminus ( \NN_0\times \ZZ_{\leqslant 4})$.
The following result gives a recursive description of the spaces $H^n_m$.

\begin{prop}
\label{H_{-1}}
(\cite{es zl}, Cor. 5.3)
For integers $m\leqslant 1$ and $n\in\NN_0$, we have 
\begin{equation}\label{euqation:Hnm coho}	
\begin{split}
H_{m}^n\cong
\begin{cases}
\omega^*_{\frac{1-m}{2}}H_{1}^{n+2m-2},  & \text{if $m$ is odd}, 
\\
\omega^*_{-\frac{m}{2}}H_0^{n+2m},  & \text{if $m$ is even}.
\end{cases} 
\end{split}
\end{equation}
\end{prop}

Given elements $x\in \mathcalboondox{B}_{m}$ and $y\in \mathcalboondox{B}_{n}^{!*}$, the symbol $y|x$ will denote the linear map in $K^n=\operatorname{Hom}_{\Bbbk}((A^!_{-n})^*,A)$, which maps $y$ to $x$ and sends the other usual basis elements of $(A^!_{-n})^*$ to zero. 
See \cite{es zl}, Cor. 5.4, for specific representatives of the cohomology classes of a basis of $H^n_m$ for $(n,m)\in \NN_0\times \ZZ_{\leqslant 4}$.

The algebra structure of the Hochschild cohomology $\HH^{\bullet}(A)$ (with the multiplication given by the cup product) is described as follows. 
\begin{thm}
\label{thm: algebra structure of cohomology}
(\cite{es zl}, Cor. 6.11)
The Hochschild cohomology $\HH^{\bullet}(A)$ is isomorphic to the quotient of the free graded-commutative (for the cohomological degree) $\Bbbk$-algebra generated by fourteen elements (with fixed cohomological degrees and internal degrees) modulo the ideal generated by the elements given in \cite{es zl}, (6.5).
\end{thm}

\begin{rk}
To avoid any confusion, we remark that the definition of cup product on Hochschild cohomology in the previous theorem is the one given in \cite{MR161898}, Section 7. 
Explicitly, at the level of the bar resolution it is given by $(f  \hskip 0.2mm \mathsmaller{\cup} \hskip 0.2mm g)(a_{1},\dots,a_{m+n}) = f(a_{1},\dots,a_{m}) g(a_{m+1},\dots,a_{m+n})$,    
for $a_{1},\dots,a_{m+n} \in A$, $f \in \operatorname{Hom}_{\Bbbk}(A^{\otimes m},A)$ and $g \in \operatorname{Hom}_{\Bbbk}(A^{\otimes n},A)$. 
A different convention, in the spirit of Koszul's sign rule, includes a sign $(-1)^{m n}$ (see \cite{Sarah}, Def. 1.3.1 and Rk. 1.3.3). 
To reduce space we will denote the cup product simply by juxtaposition. 
\end{rk}

Moreover, 
the fourteen generators of $\HH^{\bullet}(A)$ mentioned in Theorem \ref{thm: algebra structure of cohomology} are represented in $\operatorname{H}^{\bullet}(Q^{\bullet})$ by the following cocycles: 
$X_1=\epsilon^!|(ab+ba)$, 
$X_2=\epsilon^!|(ab+bc-ac)$, 
$X_3=\epsilon^!|abac$, 
$X_4=\alpha|bac$, 
$X_5=\beta|abc$, 
$X_6=\gamma|aba$, 
$X_7=\alpha|(aba-abc)$, 
$X_8=\alpha|a+\beta|b+\gamma|c$, 
$X_9=\alpha_2|1$, 
$X_{10}=\beta_2|1$, 
$X_{11}=\gamma_2|1$, 
$X_{12}=(\alpha\beta+\alpha\gamma)|1$, 
$X_{13}=\alpha_3|a+\beta_3|b+\gamma_3|c$ and $X_{14}=\omega^*_1\epsilon^!|1$.
Let $Y_i\in \Hom_{A^e}(P^b_n,A)$ be the element associated to $X_i$ via the isomorphism \eqref{X Y} for $i\in \llbracket 1,14\rrbracket $.
In what follows and to simplify our notation, given a cocycle $\phi$, we will use the same symbol $\phi$ for its cohomology class.

Let $i_{\bullet}:P^b_{\bullet}\to B_{\bullet}(A)$ be a morphism of complexes of $A$-bimodules lifting $\id_A$.
It is clear that $i_0:A\otimes (A^!_{0})^* \otimes A\to A\otimes A$ and $i_1:A\otimes (A^!_{-1})^* \otimes A\to A^{\otimes 3}$ can be chosen as follows 
\begin{equation}
\label{eq:morphism}
    \begin{split}
      i_0(1|\epsilon^!|1)=1|1, \quad 
      i_1(1|\alpha|1)=-1|a|1, \quad 
      i_1(1|\beta|1)=-1|b|1, \quad 
      i_1(1|\gamma|1)=-1|c|1.
    \end{split}
\end{equation}


\section{\texorpdfstring{Gerstenhaber brackets on Hochschild cohomology of $\FK(3)$}{Gerstenhaber brackets on Hochschild cohomology of FK(3)}}
\label{section:Gerstenhaber brackets on Hochschild cohomology}

\subsection{\texorpdfstring{Gerstenhaber brackets of $\HH^0(A)$ with $\HH^n(A)$}{Gerstenhaber brackets of HH0(A) with HHn(A)}}
\label{subsection:Gerstenhaber brackets of HH^0(A) with HH^n(A)} 

In this subsection we are going to utilize the method introduced in Subsection \ref{subsection:Gerstenhaber brackets of HH^0} to compute the Gerstenhaber bracket of $X_{i}$ for $i \in \llbracket 1, 14 \rrbracket$ with the elements $X_{1}, X_{2}, X_{3}$ in $\HH^{0}(A)$. 
To wit, for every element $X_{i}$ with $i \in \llbracket 1, 3 \rrbracket$, we find the associated element $\rho$ in the center $\Z(A)$ such that $\ell_{\rho} i_{0} = X_{i}$, provide the corresponding self-homotopy $h^{\rho}_{\bullet}$ satisfying \eqref{eq:eta} and then compute the respective Gerstenhaber brackets by means of Theorem \ref{HH0}.  

We remark first that $[X_i,1]=0$ for $i \in \llbracket 1, 14 \rrbracket$, since  
$h^1_{\bullet}=0$.
On the other hand, Definition \ref{Gerstenhaber bracket at chain level} tells us that $[X_i,X_j]=0$ for $i,j\in\llbracket 1,3 \rrbracket$. 
The proof of the following three results is a lengthy but straightforward computation. 

\begin{fact}
\label{fact:X1}
Let $\rho = ab+ba \in \Z(A)$. 
Then, there is a self-contracting homotopy $h^{\rho}_{\bullet}$ satisfying \eqref{eq:eta} such that 
\begin{align*}
    h^{\rho}_0(1|\epsilon^!|1)&=-b|\alpha|1-a|\beta|1-1|\alpha|b-1|\beta|a,
    \\
    h^{\rho}_n(1|\alpha_n|1)&=(-1)^{n+1}b|\alpha_{n+1}|1-1|\alpha_{n+1}|b,
    \\
    h^{\rho}_n(1|\beta_n|1)&=(-1)^{n+1}a|\beta_{n+1}|1-1|\beta_{n+1}|a
\end{align*}
for $n\in\NN$, and  
\begin{align*}
h^{\rho}_1(1|\gamma|1)&=b|\alpha_2|1+a|\beta_2|1+a|\alpha\beta|1+b|\alpha\gamma|1-1|\alpha_2|b-1|\beta_2|a-1|\alpha\beta|b-1|\alpha\gamma|a,
    \\
h^{\rho}_2(1|\gamma_2|1)&=a|\gamma_3|1+b|\gamma_3|1+c|\alpha_2\beta|1+c|\alpha\beta_2|1+1|\gamma_3|a+1|\gamma_3|b+1|\alpha_2\beta|c+1|\alpha\beta_2|c,
    \\
h^{\rho}_2(1|\alpha\beta|1)&=-b|\alpha_3|1-c|\beta_3|1-a|\alpha_2\gamma|1-1|\alpha_3|c-1|\beta_3|a-1|\alpha_2\gamma|b,
    \\
h^{\rho}_2(1|\alpha\gamma|1)&=-c|\alpha_3|1-a|\beta_3|1-b|\alpha_2\gamma|1-1|\alpha_3|b-1|\beta_3|c-1|\alpha_2\gamma|a.
\end{align*}
\end{fact}

\begin{fact}
\label{fact:X2}
Let $\rho = ab+bc-ac \in \Z(A)$. 
Then, there is a self-contracting homotopy $h^{\rho}_{\bullet}$ satisfying \eqref{eq:eta} such that 
\begin{align*}
    h^{\rho}_0(1|\epsilon^!|1)&=c|\alpha|1+a|\gamma|1+1|\alpha|c+1|\gamma|a,
    \\
    h^{\rho}_n(1|\alpha_n|1)&=(-1)^n c|\alpha_{n+1}|1+1|\alpha_{n+1}|c,
    \\
    h^{\rho}_n(1|\gamma_n|1)&=(-1)^n a|\gamma_{n+1}|1+1|\gamma_{n+1}|a
\end{align*}
for $n\in \NN$, and 
\begin{align*}
    h^{\rho}_1(1|\beta|1)&=-c|\alpha_2|1-a|\gamma_2|1-c|\alpha\beta|1-a|\alpha\gamma|1+1|\alpha_2|c+1|\gamma_2|a+1|\alpha\beta|a+1|\alpha\gamma|c,
    \\
    h^{\rho}_2(1|\beta_2|1)&= -a|\beta_3|1-c|\beta_3|1-b|\alpha_2\gamma|1-b|\alpha\beta_2|1-1|\beta_3|a-1|\beta_3|c-1|\alpha_2\gamma|b-1|\alpha\beta_2|b,
    \\
    h^{\rho}_2(1|\alpha\beta|1)&=b|\alpha_3|1+a|\gamma_3|1+c|\alpha_2\beta|1+1|\alpha_3|c+1|\gamma_3|b+1|\alpha_2\beta|a,
    \\
    h^{\rho}_2(1|\alpha\gamma|1)&= c|\alpha_3|1+b|\gamma_3|1+a|\alpha_2\beta|1+1|\alpha_3|b+1|\gamma_3|a+1|\alpha_2\beta|c.
\end{align*}
\end{fact}

\begin{fact}
\label{fact:X3}
Let $\rho = abac \in \Z(A)$. 
Then, there is a self-contracting homotopy $h^{\rho}_{\bullet}$ satisfying \eqref{eq:eta} such that 
\begin{align*}
    h^{\rho}_0(1|\epsilon^!|1)&=-aba|\gamma|1-ab|\alpha|c-a|\beta|ac-1|\alpha|bac,
    \\
    h^{\rho}_1(1|\alpha|1)&=aba|\alpha\beta|1-ab|\alpha_2|b-ba|\beta_2|c+c|\alpha_2|bc+b|\beta_2|ac+b|\alpha\beta|bc-1|\alpha_2|bac-1|\alpha\beta|abc,
    \\
    h^{\rho}_1(1|\beta|1)&=aba|\alpha\gamma|1-2ab|\alpha_2|c-ac|\alpha_2|a-ab|\alpha\beta|a+a|\alpha_2|bc-a|\beta_2|ab-a|\beta_2|bc+c|\beta_2|ac
    \\
    & \phantom{= \;} 
    +a|\alpha\gamma|ac-1|\alpha\gamma|bac,
    \\
    h^{\rho}_1(1|\gamma|1)&=2aba|\gamma_2|1-ba|\alpha\beta|c+b|\alpha_2|bc-a|\gamma_2|ac-c|\alpha\beta|ac-1|\alpha\gamma|abc,
    \\
    h^{\rho}_2(1|\alpha_2|1)&=bac|\alpha_3|1+bc|\beta_3|a-ba|\beta_3|c+ba|\alpha_2\gamma|a-b|\alpha_3|ab-b|\alpha_3|bc+c|\alpha_3|ba+a|\beta_3|ac
    \\
    & \phantom{= \;} 
    +c|\beta_3|bc+a|\alpha_2\gamma|bc+b|\alpha_2\gamma|ba-2|\alpha_3|bac,
    \\
    h^{\rho}_2(1|\beta_2|1)&=abc|\beta_3|1-2ab|\alpha_3|c+ac|\alpha_3|b+ab|\beta_3|a-bc|\beta_3|a+ab|\alpha_2\gamma|b-ba|\alpha_2\gamma|a+b|\alpha_3|ab
    \\
    & \phantom{= \;} 
    +2b|\alpha_3|bc-c|\alpha_3|ba+c|\alpha_3|ac-a|\beta_3|ac-b|\alpha_2\gamma|ba+b|\alpha_2\gamma|ac-1|\alpha_3|bac
    \\
    & \phantom{= \;} 
    -2|\beta_3|abc 
    -1|\alpha_2\gamma|aba,
    \\
    h^{\rho}_2(1|\gamma_2|1)&=-3aba|\gamma_3|1+ba|\beta_3|c-ab|\alpha\beta_2|c-b|\alpha_3|bc-a|\beta_3|ac-c|\beta_3|bc+b|\gamma_3|ac-a|\alpha_2\gamma|bc
    \\
    & \phantom{= \;} 
    +1|\alpha_3|bac-1|\alpha_2\beta|abc,
    \\
    h^{\rho}_2(1|\alpha\beta|1)&=-2aba|\alpha\beta_2|1-ac|\alpha_3|c-bc|\beta_3|b-2ba|\beta_3|a-ab|\alpha_2\gamma|c-ba|\alpha_2\gamma|b-a|\alpha_3|bc
    \\
    & \phantom{= \;} 
    +b|\alpha_3|ba
    +2b|\alpha_3|ac+b|\beta_3|ab-c|\beta_3|ac+a|\gamma_3|ac-c|\gamma_3|ab+c|\alpha_2\beta|ac-a|\alpha_2\gamma|ac
    \\
    & \phantom{= \;} 
    -a|\alpha\beta_2|bc
    -1|\alpha_2\gamma|bac-1|\alpha\beta_2|abc,
    \\
    h^{\rho}_2(1|\alpha\gamma|1)&=-abc|\alpha_3|1-2aba|\alpha_2\beta|1-3ab|\alpha_3|b-ab|\beta_3|c-2bc|\beta_3|c-2ba|\alpha_2\gamma|c-ba|\alpha\beta_2|a
    \\
    & \phantom{= \;} 
    +2a|\alpha_3|ba-c|\alpha_3|bc-b|\beta_3|ac-b|\gamma_3|ab-a|\alpha_2\beta|ab-b|\alpha_2\gamma|bc+c|\alpha_2\gamma|ac
    \\
    & \phantom{= \;} 
    -c|\alpha\beta_2|ab
    -1|\alpha_2\beta|bac-2|\alpha_2\gamma|abc.
\end{align*}
\end{fact}

Using the previous results together with Theorem \ref{HH0} we obtain the Gerstenhaber bracket between $X_{i}$ for $i \in \llbracket 1, 14 \rrbracket$ and $X_{1},X_{2},X_{3}$.
\begin{prop}
\label{prop:bracket-H0-Xi} 
The Gerstenhaber bracket on $\HH^{\bullet}(A)$ of $X_{i}$ for $i \in \llbracket 1, 14 \rrbracket$ with an element $X_{j}$ for $j \in \llbracket 1, 3 \rrbracket$ is given by 
\begin{equation}
\label{X_i X_1}
    [X_{i},X_{1}] = \begin{cases} 
    -2 X_{1}, &\text{if $i = 8$,}
    \\
    -4 X_{1} (X_{9} + X_{10}), &\text{if $i = 13$,}
    \\
    0, &\text{if $i \in \llbracket 1, 14 \rrbracket \setminus \{ 8, 13\}$,}
\end{cases}
\end{equation}
\begin{equation}
\label{X_i,X_2}
    [X_{i},X_{2}] = \begin{cases} 
    -2 X_{2}, &\text{if $i = 8$,}
    \\
    -4 X_{1} X_{10}, &\text{if $i = 13$,}
    \\
    0, &\text{if $i \in \llbracket 1, 14 \rrbracket \setminus \{ 8, 13\}$,}
\end{cases}
\end{equation}
and 
\begin{equation}
\label{X_i,X_3}
    [X_{i},X_{3}] = \begin{cases} 
    0, &\text{if $i \in \llbracket 1, 7 \rrbracket $,}
    \\
    -4 X_{3}, &\text{if $i = 8$,}
    \\
    -2X_{i-5}, &\text{if $i=9,10$,}
    \\
    2X_{6}, &\text{if $i=11$,}
    \\
    2X_7-X_1X_8+X_2X_8, &\text{if $i = 12$,}
    \\
    -4 X_{3} (X_{9} + X_{10}+X_{11}), &\text{if $i = 13$,}
    \\
    X_{13}-(2/3)X_8(X_9+X_{10}+ X_{11}), &\text{if $i = 14$.}
\end{cases}
\end{equation}
\end{prop}
\begin{proof}
Note that $\ell_{ab+ba}i_0=Y_1$, $\ell_{ab+bc-ac}i_0=Y_2$ and $\ell_{abac}i_0=Y_3$. 
Applying Theorem \ref{HH0} together with Facts \ref{fact:X1}, \ref{fact:X2} and \ref{fact:X3}, we get the brackets 
\begin{equation}
\label{X_1-}
    [X_{i},X_{1}] = \begin{cases} 
    -2 X_{1}, &\text{if $i = 8$,}
    \\
    -(\alpha_2+\beta_2+\gamma_2)|(ab+ba)-\alpha\beta|ba-\alpha\gamma|ab, &\text{if $i = 13$,}
    \\
    0, &\text{if $i \in \llbracket 1, 12 \rrbracket \setminus \{ 8\}$,}
\end{cases}
\end{equation}
\begin{equation}
\label{X_2-}
    [X_{i},X_{2}] = \begin{cases} 
    -2 X_{2}, &\text{if $i = 8$,}
    \\
    (\alpha_2+\beta_2+\gamma_2)|(ac-ab-bc)+\alpha\beta|ac-\alpha\gamma|(ab+bc) , &\text{if $i = 13$,}
    \\
    0, &\text{if $i \in \llbracket 1, 12 \rrbracket \setminus \{ 8\}$,}
\end{cases}
\end{equation}
and 
\begin{equation}
\label{X_3-}
    [X_{i},X_{3}] = \begin{cases} 
    0, &\text{if $i \in \llbracket 1, 7 \rrbracket $,}
    \\
    -4 X_{3}, &\text{if $i = 8$,}
    \\
    -2X_{i-5}, &\text{if $i=9,10$,}
    \\
    2X_{6}, &\text{if $i=11$,}
    \\
    \alpha|(aba-abc)-\beta|bac-\gamma|bac , &\text{if $i = 12$,}
    \\
    -4(\alpha_2+\beta_2+\gamma_2)|abac , &\text{if $i = 13$.}
\end{cases}
\end{equation}
Indeed, this was simply done by computing 
$[Y_i, Y_1]=Y_i h^{ab+ba}_{ \mathfrak{h} (Y_i)-1 }$, 
$[Y_i, Y_2]=Y_i h^{ab+bc-ac}_{ \mathfrak{h} (Y_i)-1 }$, and 
$[Y_i, Y_3]=Y_i h^{abac}_{ \mathfrak{h} (Y_i)-1 }$, 
where $\mathfrak{h} (Y_i)$ denotes the cohomological degree of $Y_i$ for $i\in \llbracket 1,13 \rrbracket$, 
and by transport of structures. 
Note that the vanishing of $[X_{i},X_{3}]$ for $i \in \llbracket 4,7 \rrbracket$ also follows from a simple degree argument using Corollary \ref{cor:gr Hom} together with \cite{es zl}, Cor. 5.9. 
The latter two results also tell us that $[X_{14},X_j]=0$ (or $[Y_{14},Y_j]=0$) for $j=1,2$, by degree reasons. 
This result also follows from noting that $h^{ab+ba}_3$ is of internal degree $2$, so $h^{ab+ba}_3(1|u|1)$ is of internal degree $5$ for any $u\in \mathcalboondox{B}^{!*}_3$, which implies that 
$Y_{14}(h^{ab+ba}_3(1|u|1))$ vanishes, since $Y_{14}$ vanishes on any homogeneous element of internal degree strictly less than $6$. 
Hence, $Y_{14}h^{ab+ba}_3=0$. 
We get $Y_{14}h^{ab+bc-ac}_3=0$ for the same reason. 

Next, we compute $\varphi=[Y_{14},Y_3]=Y_{14}h^{abac}_3$. 
By \eqref{eq:eta}, the map $h^{abac}_3 :P^b_3\to P^b_4$ satisfies $\delta^b_4 h^{abac}_3=\eta_3-h^{abac}_2\delta^b_3$.
It is easy to check that 
\allowdisplaybreaks
\begin{align*}
    (\eta_3-h^{abac}_2\delta^b_3)(1|\alpha_3|1)&=-bac|\alpha_3|a+abc|\beta_3|a-aba|\beta_3|c+aba|\alpha_2\gamma|a+v_{\alpha_3},
    \\
    (\eta_3-h^{abac}_2\delta^b_3)(1|\beta_3|1)&=-2aba|\alpha_3|c+bac|\alpha_3|b+aba|\beta_3|a-abc|\beta_3|b+aba|\alpha_2\gamma|b+v_{\beta_3},
    \\
    (\eta_3-h^{abac}_2\delta^b_3)(1|\gamma_3|1)&=abc|\beta_3|c+3aba|\gamma_3|c-bac|\alpha\beta_2|c+v_{\gamma_3},
    \\
    (\eta_3-h^{abac}_2\delta^b_3)(1|\alpha_2\beta|1)&=abc|\alpha_3|a-4bac|\alpha_3|b-2aba|\beta_3|a-abc|\beta_3|b+bac|\beta_3|c+3aba|\gamma_3|b
    \\
    & \phantom{= \;} 
    +2aba|\alpha_2\beta|a-aba|\alpha_2\gamma|b-2abc|\alpha_2\gamma|c-abc|\alpha\beta_2|a+aba|\alpha\beta_2|c
    \\
    & \phantom{= \;} 
    +v_{\alpha_2\beta},
    \stepcounter{equation}\tag{\theequation}\label{eq:auxiliary-eta-hdelta}
    \\
    (\eta_3-h^{abac}_2\delta^b_3)(1|\alpha_2\gamma|1)&=-4bac|\alpha_3|c+bac|\beta_3|a-4abc|\beta_3|c+2aba|\alpha_2\beta|b+bac|\alpha_2\gamma|b
    \\
    & \phantom{= \;} 
    -3aba|\alpha_2\gamma|c+aba|\alpha\beta_2|a+v_{\alpha_2\gamma},
    \\
    (\eta_3-h^{abac}_2\delta^b_3)(1|\alpha\beta_2|1)&=-3aba|\alpha_3|b+2abc|\alpha_3|c-4abc|\beta_3|a+bac|\beta_3|b+3aba|\gamma_3|a
    \\
    & \phantom{= \;} 
    +2aba|\alpha_2\beta|c-aba|\alpha_2\gamma|a-abc|\alpha_2\gamma|b-bac|\alpha_2\gamma|c+2aba|\alpha\beta_2|b
    \\
    & \phantom{= \;} 
    +v_{\alpha\beta_2},
\end{align*}
where $v_{u} \in \oplus_{j\in \llbracket0,4 \rrbracket\setminus \{ 3 \} } (A_j\otimes (A^!_{-3})^{*}\otimes A_{4-j})$ for $u\in \mathcalboondox{B}^{!*}_3$.
By degree reasons, the element $h^{abac}_3(1|u|1)$ for $u\in \mathcalboondox{B}^{!*}_3$ is of the form  
\[  
h^{abac}_3(1|u|1)=B_u+\omega_1 (\lambda^u_1 a|\epsilon^!|1+\lambda^u_2 1|\epsilon^!|a+\lambda^u_3b|\epsilon^!|1+\lambda_4^u 1|\epsilon^!|b +\lambda^u_5 c|\epsilon^!|1+\lambda^u_6 1|\epsilon^!|c ), 
\]
where $B_u\in K^b_4$ has internal degree $7$  
and $\lambda^u_i\in \Bbbk$ for $i\in \llbracket 1,6 \rrbracket$. 
Therefore, 
\begin{equation}
\label{eq:B lambda X3}
    \begin{split}
       (\eta_3-h^{abac}_2\delta^b_3)(1|u|1)
       &=\delta^b_4 h^{abac}_3(1|u|1)
       \\
       &=d^b_4(B_u)+\lambda^u_1 f^b_0(a|\epsilon^!|1)+\lambda^u_2 f^b_0(1|\epsilon^!|a)+\lambda^u_3 f^b_0(b|\epsilon^!|1)+\lambda^u_4 f^b_0(1|\epsilon^!|b)
       \\
       & \phantom{= \;} 
       +\lambda^u_5 f^b_0(c|\epsilon^!|1)+\lambda^u_6 f^b_0(1|\epsilon^!|c).  
    \end{split}
\end{equation}
Using the explicit expression of the differential $d^b_4$ given in \cite{es zl}, Fact 3.1, together with an elementary computation we see that, given any homogeneous element $B \in K_{4}^{b}$ of internal degree $7$, the coefficients of $aba|\gamma_3|a$ and $aba|\alpha_2\gamma|a$ in $d^b_4(B)$ are equal, 
the coefficients of $aba|\gamma_3|b$ and $aba|\alpha_2\gamma|b$ in $d^b_4(B)$ coincide, 
and the coefficients of $abc|\beta_3|c$ and $abc|\alpha_2\beta|c$ in $d^b_4(B)$ are also the same.
Comparing the coefficients of $aba|\gamma_3|a$ and $aba|\alpha_2\gamma|a$ in both sides of the equation \eqref{eq:B lambda X3}, where the left member is explicitly given by \eqref{eq:auxiliary-eta-hdelta} and the right member is computed using Fact \ref{fact:fb0degree1}, 
we obtain 
\[ \lambda^{\alpha_3}_1+\lambda^{\alpha_3}_2=1/3, \quad
 \lambda^{\alpha\beta_2}_1+\lambda^{\alpha\beta_2}_2=-4/3, \quad 
\lambda^{u}_1+\lambda^{u}_2=0   \]  
for $u\in \mathcalboondox{B}^{!*}_3 \setminus \{ \alpha_3,\alpha\beta_2  \} $. 
Similarly, comparing the coefficients of $aba|\gamma_3|b$ and $aba|\alpha_2\gamma|b$ in both sides of the equation \eqref{eq:B lambda X3}, where the left member is explicitly given by \eqref{eq:auxiliary-eta-hdelta} and the right member is computed using Fact \ref{fact:fb0degree1},
we obtain 
\[  
\lambda^{\beta_3}_3+\lambda^{\beta_3}_4=1/3, \quad 
\lambda^{\alpha_2\beta}_3+\lambda^{\alpha_2\beta}_4=-4/3, \quad 
\lambda^{u}_3+\lambda^{u}_4=0
\]  
for $u\in \mathcalboondox{B}^{!*}_3 \setminus \{ \beta_3,\alpha_2\beta  \} $. 
Comparing the coefficients of $abc|\beta_3|c$ and $abc|\alpha_2\beta|c$ in both sides of the equation \eqref{eq:B lambda X3}, where the left member is explicitly given by \eqref{eq:auxiliary-eta-hdelta} and the right member is computed using Fact \ref{fact:fb0degree1},
we obtain 
\[  
\lambda^{\gamma_3}_5+\lambda^{\gamma_3}_6=1/3, \quad 
\lambda^{\alpha_2\gamma}_5+\lambda^{\alpha_2\gamma}_6=-4/3, \quad \lambda^{u}_5+\lambda^{u}_6=0
\]  
for $u\in \mathcalboondox{B}^{!*}_3 \setminus \{ \gamma_3,\alpha_2\gamma  \} $. 
Then $\varphi(1|u|1)=Y_{14}h^{abac}_3(1|u|1)$ for $u\in \mathcalboondox{B}^{!*}_3 $ is given by 
\begin{equation}\label{eq:0}
    \begin{aligned}
        \varphi(1|\alpha_3|1)&=(1/3)a, \quad &
        \varphi (1|\beta_3|1)&=(1/3)b, \quad &
        \varphi (1|\gamma_3|1)&=(1/3)c, 
        \\
        \varphi (1|\alpha_2\beta|1)&=-(4/3)b, \quad &
        \varphi (1|\alpha_2\gamma|1)&=-(4/3)c, \quad &
        \varphi (1|\alpha\beta_2|1)&= -(4/3)a.
    \end{aligned}
\end{equation}
Hence, 
$[X_{14}, X_3]=(1/3)(\alpha_3|a+\beta_3|b+\gamma_3|c)-(4/3)(\alpha_2\beta|b+\alpha_2\gamma|c+\alpha\beta_2|a)$.

We now note the following identities, 
\begin{equation}
\label{eq:some cup products1}
\begin{aligned}
    \alpha_2|(ab+ba)&=X_1X_9, \quad &
    \beta_2|(ab+ba)&=X_1X_{10}, \quad &
    \alpha|aba+\beta|bac&=(1/2)(X_1X_8-X_2X_8), 
    \\
    \alpha_2|abac &=X_3X_9, \quad & 
    \beta_2|abac &=X_3X_{10}, \quad & 
    \gamma_2|abac &=X_3X_{11}
    \end{aligned}
\end{equation}
and 
\begin{equation}
\label{eq:some cup products2}
\begin{split}
    (\alpha_3-\alpha\beta_2)|a &=(1/2)\{ X_{13}+X_8(X_9-X_{10}-X_{11})\}, 
    \\
    (\beta_3-\alpha_2\beta)|b& =(1/2)\{ X_{13}+X_8(X_{10}-X_{9}-X_{11}) \}, 
    \\
    (\gamma_3-\alpha_2\gamma)|c &=(1/2)\{ X_{13}+X_8(X_{11}-X_{9}-X_{10})\},  
    \end{split}
\end{equation}
given in \cite{es zl}, (6.2).  
Using the previous equalities as well as the coboundaries 
$g^2_{j,2}\in \tilde{\mathfrak{B}}^2_2$ for $j\in\llbracket 1,8 \rrbracket \setminus \{ 4,5 \}$ and $e^1_{1,3}=\alpha|(aba+abc)+(\beta-\gamma)|bac \in  \tilde{\mathfrak{B}}^1_3$ of the sets $\tilde{\mathfrak{B}}^2_2$ and $\tilde{\mathfrak{B}}^1_3$ given in \cite{es zl}, Subsubsection 5.3.1, we get
\begin{align*}
    [X_{13},X_1]
    &=-(\alpha_2+\beta_2+\gamma_2)|(ab+ba)-\alpha\beta|ba-\alpha\gamma|ab-3g^2_{1,2}-3g^2_{2,2}-2g^2_{3,2}+g^2_{8,2}
    \\
    &=-4(\alpha_2+\beta_2)|(ab+ba)
    =-4X_1(X_9+X_{10}),
    \\
    [X_{13},X_2]
    &=(\alpha_2+\beta_2+\gamma_2)|(ac-ab-bc)+\alpha\beta|ac-\alpha\gamma|(ab+bc) -g^2_{1,2}-2g^2_{2,2}-g^2_{3,2}+g^2_{6,2}
    \\
    & \phantom{= \;} 
    -g^2_{7,2}
    +g^2_{8,2}
    \\
    &=-4\beta_2|(ab+ba) 
    =-4X_1X_{10}, 
    \\
    [X_{12},X_3]&=\alpha|(aba-abc)-\beta|bac-\gamma|bac-e^1_{1,3}
    =2\alpha|(aba-abc)-2(\alpha|aba+\beta|bac)
    \\
    &=2X_7-X_1X_8+X_2X_8,   
    \\
    [X_{13},X_3]&= -4(\alpha_2+\beta_2+\gamma_2)|abac
    =-4X_3(X_9+X_{10}+X_{11}), 
    \\
    [X_{14},X_3]&=(1/3)(\alpha_3|a+\beta_3|b+\gamma_3|c)-(4/3)(\alpha_2\beta|b+\alpha_2\gamma|c+\alpha\beta_2|a)
    \\
    &=X_{13}-(2/3)X_8(X_9+X_{10}+ X_{11}).
\end{align*}
The proposition is thus proved. 
\end{proof}

\subsection{\texorpdfstring{Gerstenhaber brackets of $\HH^1(A)$ with $\HH^n(A)$}{Gerstenhaber brackets of HH1(A) with HHn(A)}}
\label{subsection:Gerstenhaber brackets of HH^1(A)} 

In this subsection we are going to utilize the method recalled in Subsection \ref{subsection:Gerstenhaber brackets of HH^1} to compute the Gerstenhaber bracket of $X_{i}$ for $i \in \llbracket 4, 8 \rrbracket$ with the elements $X_{j}$ for $j \in \llbracket 1, 14 \rrbracket$. 

Let $\rho: A\to A$ be a derivation of $A$.
By \cite{Mariano}, Lemma 1.3, 
the $\rho^e$-lifting $\rho_{\bullet}=\{ \rho_n :P^b_n\to P^b_n\}_{n\in\NN_0}$ of $\rho$ to $(P^b_{\bullet},\delta^b_{\bullet})$ 
exists, and it can be chosen in such a way that
\begin{equation}
\label{eq:lifting}
    \begin{split}
        \rho_0(x|\epsilon^!|y)=\rho(x)|\epsilon^!|y+x|\epsilon^!|\rho(y) \text{ and }
        \rho_n(\omega_ix|u|y)=xq_{\omega_iu}y+\omega_i\rho(x)|u|y+\omega_ix|u|\rho(y) ,
    \end{split}
\end{equation}
for all $x,y\in A$, $n\in\NN$, $i\in \llbracket 0,\lfloor n/4 \rfloor \rrbracket$ and $u\in  \mathcalboondox{B}^{!*}_{n-4i}$, where $q_{\omega_iu}\in P^b_n$ satisfies that $\delta^b_n(q_{\omega_i u})=\rho_{n-1}\delta^b_n(\omega_i1|u|1)$.
To reduce space, we will usually write $q_{u}$ instead of $q_{\omega_0 u}$. 
As recalled in Subsection \ref{subsection:Gerstenhaber brackets of HH^1}, given $\phi\in \HH^n(A)$, the Gerstenhaber bracket $[G(\rho)i_1,\phi]\in \HH^n(A)$ is given by the cohomology class of $\rho\phi-\phi\rho_n $. 

In what follows, we consider a set of derivations of $A$ whose classes give a basis of $\operatorname{HH}^{1}(A)$ and for each of them we will provide some of the corresponding elements $q_{\omega_i u}$ satisfying \eqref{eq:lifting}. 
Then, we shall compute the respective Gerstenhaber brackets by means of Theorem \ref{thm:HH1 HHn}. 

The proof of the following result follows immediately from the statement. 
\begin{prop}
\label{prop:X_8}
Let $\rho:A\to A$ be the derivation of $A$ defined by $\rho(x)= \deg(x)x$ for $x\in \mathcalboondox{B}$.
Then $\rho_{\bullet}$ defined by $\rho_n(\omega_i x|u|y)=(\deg(x)+\deg(y)+n+2i)\omega_ix|u|y$ for $x,y\in \mathcalboondox{B} $, $i\in \llbracket 0,\lfloor n/4 \rfloor \rrbracket$, $u\in  \mathcalboondox{B}^{!*}_{n-4i}$ and $n\in\NN_0$ is a $\rho^e$-lifting of $\rho$.
Note that $\deg(x)+\deg(y)+n+2i$ is the internal degree of $\omega_i x|u|y$.
Since $G(\rho)i_1=-X_8$, 
the Gerstenhaber bracket $[X_8,\phi]\in\HH^n(A)$ for $\phi\in \HH^n(A)$ is given by the cohomology class  $-\mathfrak{a}(\phi)\phi$, where $\mathfrak{a}(\phi)$ is the internal degree of $\phi$. 
Hence, 
\begin{equation}
\label{eq:X8 Xn}
[X_{8},X_{j}] = \begin{cases} 
    -2X_j, &\text{if $j \in \llbracket 1, 7 \rrbracket \setminus \{ 3 \}$,}
    \\
    -4X_{3}, &\text{if $j = 3$,}
    \\
    0 , &\text{if $j = 8$,}
    \\
    2X_{j}, &\text{if $j \in \llbracket 9, 13 \rrbracket $,}
    \\
    6X_{14}, &\text{if $j = 14$.}
\end{cases}
\end{equation}
\end{prop}

The proof of Facts \ref{fact:-X4}, \ref{fact:-X5}, \ref{fact:-X6} and \ref{fact:-X7} below is a lengthy but straightforward computation. 

\begin{fact}
\label{fact:-X4}
Let $\rho=\rho^{4} : A\to A$ be the derivation of $A$ defined by $\rho^{4}(a)=bac$ and $\rho^{4}(x)=0$ for $x\in \mathcalboondox{B}\setminus \{a\} $.  
Then the elements $q_{\omega_iu}=q_{\omega_iu}^{4} \in P^b_n$ in \eqref{eq:lifting} can be chosen as follows. 
First, $q_{\beta_n}^{4}=q_{\gamma_n}^{4}=0$ for $n\in\NN$. 
Moreover, 
\begin{align*}
    q_{\alpha}^{4}&=ba|\gamma|1+b|\alpha|c+1|\beta|ac,
    \\
    q_{\alpha_2}^{4}&=ba|\alpha\beta|1-b|\alpha_2|b-c|\alpha_2|c-b|\alpha\beta|c+1|\alpha\beta|ac,
    \\
    q_{\alpha\beta}^{4}&=ab|\gamma_2|1-ab|\alpha_2|1+ba|\alpha\gamma|1-2b|\alpha_2|c-c|\alpha_2|a-b|\beta_2|c-a|\alpha\beta|c-b|\alpha\beta|a
     +1|\alpha_2|bc
     \\
     & \phantom{= \;} 
     -1|\beta_2|ab+1|\alpha\gamma|ac,
    \\
    q_{\alpha\gamma}^{4}&=ba|\gamma_2|1+1|\beta_2|ac,
    \\
    q_{\alpha_3}^{4}&=bc|\beta_3|1+ba|\alpha_2\gamma|1+b|\alpha_3|c+c|\alpha_3|b+b|\beta_3|a-c|\beta_3|c-a|\alpha_2\gamma|c+b|\alpha_2\gamma|b
     -1|\beta_3|ac
     \\
     & \phantom{= \;} 
     -1|\alpha_2\gamma|bc,
    \\
    q_{\alpha_2\beta}^{4}&=ab|\gamma_3|1+ba|\alpha\beta_2|1-a|\alpha_3|b-2a|\beta_3|c-c|\beta_3|a-b|\gamma_3|c-c|\gamma_3|b-a|\alpha_2\beta|c
     -c|\alpha\beta_2|c 
     \\
     & \phantom{= \;} 
     -b|\alpha_2\gamma|c-a|\alpha_2\gamma|a+1|\alpha_3|bc+1|\gamma_3|ba-1|\alpha_2\gamma|ba-1|\alpha\beta_2|bc,
    \\
    q_{\alpha_2\gamma}^{4}&=ba|\alpha_2\beta|1+ab|\alpha\beta_2|1-2ab|\alpha_3|1-ba|\beta_3|1+a|\alpha_3|c+2b|\alpha_3|b+c|\alpha_3|a+a|\beta_3|a
     +b|\beta_3|c
     \\
     & \phantom{= \;}
     +c|\beta_3|b+a|\alpha_2\gamma|b+a|\alpha\beta_2|c-1|\alpha_3|ba-2|\beta_3|ab+1|\alpha\beta_2|ac, 
    \\
    q_{\alpha\beta_2}^{4}&=3ba|\gamma_3|1-(ab+bc)|\beta_3|1+(ab+bc)|\alpha_2\beta|1-3c|\alpha_3|b-2a|\beta_3|b-b|\beta_3|a-c|\beta_3|c
    \\
     & \phantom{= \;}
     -2a|\gamma_3|c
     -2c|\gamma_3|a-a|\alpha_2\beta|b-2c|\alpha_2\beta|c-2b|\alpha_2\gamma|b-c|\alpha_2\gamma|a-c|\alpha\beta_2|b+2|\beta_3|ac
     \\
     & \phantom{= \;}
     +1|\gamma_3|ab
     -1|\alpha_2\gamma|ab,
     \\
    q_{\omega_1 \epsilon^!}^{4}&=bac|\alpha_4|a+4abc|\beta_4|b-aba|\gamma_4|c+4ab|\alpha_3\beta|bc-4bc|\alpha_3\beta|ab+2(ba+ac)|\alpha_3\gamma|ac
    \\
     & \phantom{= \;}
    +2ba|\alpha_3\gamma|ba
    +ab|\alpha_2\beta_2|ac+bc|\alpha_2\beta_2|ba-2ba|\alpha_2\beta_2|(ab+bc) -4ab|\alpha_4|ba-2bc|\alpha_4|ba
    \\
     & \phantom{= \;}
    +(ba+ac)|\alpha_4|bc+4ab|\alpha_4|ac-2(ba+ac)|\alpha_4|ab+8bc|\beta_4|ac-10ba|\beta_4|ab-2ac|\beta_4|ab
    \\
     & \phantom{= \;}+6ab|\beta_4|ac+4(ab+bc)|\beta_4|ba-2ab|\gamma_4|ba-4bc|\gamma_4|ba+6ab|\gamma_4|ac-5ba|\gamma_4|ab
     \\
     & \phantom{= \;}
     +4ba|\gamma_4|bc
     -6ac|\gamma_4|ab+a|\alpha_4|bac+4b|\beta_4|abc-c|\gamma_4|aba-\omega_1 c|\epsilon^!|c.
 \end{align*}
\end{fact}

\begin{fact}
\label{fact:-X5}
Let $\rho=\rho^{5} : A\to A$ be the derivation of $A$ defined by $\rho^{5}(b)=abc$ and $\rho^{5}(x)=0$ for $x\in \mathcalboondox{B}\setminus \{b\} $.
Then the elements $q_{\omega_iu}=q_{\omega_iu}^{5} \in P^b_n$ in \eqref{eq:lifting} can be chosen as follows. 
First, $q_{\alpha_n}^{5}=q_{\gamma_n}^{5}=0$ for $n\in\NN$. 
Moreover, 
\begin{align*}
    q_{\beta}^{5}&=ab|\gamma|1+a|\beta|c+1|\alpha|bc,
    \\
    q_{\beta_2}^{5}&=ab|\alpha\gamma|1-a|\beta_2|a-c|\beta_2|c-a|\alpha\gamma|c+1|\alpha\gamma|bc,
    \\
    q_{\alpha\beta}^{5}&=ab|\gamma_2|1+1|\alpha_2|bc, 
    \\
    q_{\alpha\gamma}^{5}&=ab|\alpha\beta|1-bc|\alpha\beta|1-2ba|\beta_2|1-a|\alpha_2|c-2a|\beta_2|c+b|\beta_2|a-c|\beta_2|b+b|\gamma_2|a+b|\alpha\beta|b
    \\
    & \phantom{= \;}
    -a|\alpha\gamma|b
    +1|\beta_2|ac-1|\alpha_2|ba+1|\alpha\beta|bc,
    \\
    q_{\beta_3}^{5}&=(ba+ac)|\alpha_3|1+ab|\alpha_2\gamma|1+a|\alpha_3|b+b|\alpha_3|a+2a|\beta_3|c+c|\beta_3|a+b|\gamma_3|c+c|\gamma_3|b+a|\alpha_2\beta|c
    \\
    & \phantom{= \;}
    +a|\alpha_2\gamma|a+c|\alpha\beta_2|c-1|\gamma_3|ba+1|\alpha\beta_2|bc,
    \\
    q_{\alpha_2\beta}^{5}&=2ab|\gamma_3|1-2(ba+ac)|\alpha_3|1-(ab+bc)|\alpha_2\gamma|1+(ba+ac)|\alpha\beta_2|1-a|\alpha_3|b-2b|\alpha_3|a
    \\
    & \phantom{= \;}
    -2c|\beta_3|a
    -b|\gamma_3|c-c|\gamma_3|b-2a|\alpha_2\gamma|a-c|\alpha\beta_2|c+2|\alpha_3|bc+1|\gamma_3|ba-1|\alpha_2\gamma|ba,
    \\
    q_{\alpha_2\gamma}^{5}&=ba|\alpha_2\beta|1-2ba|\beta_3|1-ab|\alpha_3|1+a|\alpha_3|c+b|\alpha_3|b+c|\alpha_3|a+2a|\beta_3|a+b|\beta_3|c+c|\beta_3|b
    \\
    & \phantom{= \;}
    +b|\alpha_2\beta|c
    +b|\alpha_2\gamma|a+b|\alpha\beta_2|b-2|\alpha_3|ba-1|\beta_3|ab+1|\alpha_2\beta|bc,
    \\
    q_{\alpha\beta_2}^{5}&=ba|\gamma_3|1+ab|\alpha_2\beta|1-b|\alpha_3|c-c|\alpha_3|b-b|\beta_3|a+c|\beta_3|c-b|\alpha_2\gamma|b+b|\alpha\beta_2|c+2|\beta_3|ac
    \\
    & \phantom{= \;}
    -1|\alpha_2\gamma|ab
    +1|\alpha_2\gamma|bc,
    \\
    q_{\omega_1\epsilon^!}^{5}&=4bac|\alpha_4|a+abc|\beta_4|b-aba|\gamma_4|c+ab|\alpha_3\beta|ab+5ab|\alpha_3\beta|bc-4bc|\alpha_3\beta|ab
    \\
    & \phantom{= \;}
    +2(ba+ac)|\alpha_3\gamma|ac
    +2ba|\alpha_3\gamma|ba-ac|\alpha_3\gamma|ba-ab|\alpha_2\beta_2|ba+2bc|\alpha_2\beta_2|ba-2ba|\alpha_2\beta_2|ab
    \\
    & \phantom{= \;}
    -ba|\alpha_2\beta_2|bc
    +ac|\alpha_2\beta_2|bc
    -11ab|\alpha_4|ba+2ab|\alpha_4|ac+6ba|\alpha_4|bc+ac|\alpha_4|ab+9ac|\alpha_4|bc
    \\
    & \phantom{= \;}
    +3ab|\beta_4|ac+bc|\beta_4|ac-6ba|\beta_4|ab+ba|\beta_4|bc-3ac|\beta_4|ab-5ab|\gamma_4|ba
    +4ab|\gamma_4|ac
    \\
    & \phantom{= \;}
    -4bc|\gamma_4|ba
    -2bc|\gamma_4|ac-3ba|\gamma_4|ab+6ba|\gamma_4|bc-8ac|\gamma_4|ab+4a|\alpha_4|bac+b|\beta_4|abc
    \\
    & \phantom{= \;}
    -c|\gamma_4|aba
    -\omega_1c|\epsilon^!|c.
\end{align*}
\end{fact}

\begin{fact}
\label{fact:-X6}
Let $\rho=\rho^{6} : A\to A$ be the derivation of $A$ defined by $\rho^{6}(c)=aba$ and $\rho^{6}(x)=0$ for $x\in \mathcalboondox{B}\setminus \{c\} $.
Then the elements $q_{\omega_iu}=q_{\omega_iu}^{6} \in P^b_n$ in \eqref{eq:lifting} can be chosen as follows. First, $q_{\alpha_n}^{6}=q_{\beta_n}^{6}=0$ for $n\in\NN$. 
Moreover, 
\begin{align*}
    q_{\gamma}^{6}&=ab|\alpha|1+a|\beta|a+1|\alpha|ba,
    \\
    q_{\gamma_2}^{6}&=ba|\alpha\beta|1-b|\alpha_2|b+c|\beta_2|c+a|\gamma_2|a+c|\alpha\beta|a+a|\alpha\gamma|c+1|\alpha\gamma|ab,
    \\
    q_{\alpha\beta}^{6}&=2ab|\alpha_2|1+c|\alpha_2|a+b|\beta_2|c+b|\alpha\beta|a+1|\beta_2|ab,
    \\
    q_{\alpha\gamma}^{6}&=ba|\beta_2|1+a|\alpha_2|c+c|\beta_2|b+a|\alpha\gamma|b+2|\alpha_2|ba,
    \\
    q_{\gamma_3}^{6}&=ab|\alpha\beta_2|1-ba|\beta_3|1+a|\beta_3|a-a|\gamma_3|b-b|\gamma_3|a-b|\alpha_2\beta|c-c|\alpha_2\beta|b-b|\alpha\beta_2|b-1|\beta_3|ab
    \\
    & \phantom{= \;}
    +1|\alpha\beta_2|ba,
    \\
    q_{\alpha_2\beta}^{6}&=ac|\alpha_3|1+ab|\alpha_2\gamma|1+2c|\alpha_3|c+2a|\beta_3|c+2c|\beta_3|a+2a|\alpha_2\gamma|a+b|\alpha_2\gamma|c+c|\alpha_2\gamma|b
    \\
    & \phantom{= \;}
    +a|\alpha\beta_2|b
    +b|\alpha\beta_2|a-1|\alpha_3|(ab+bc)+1|\alpha_2\gamma|ba,
    \\
    q_{\alpha_2\gamma}^{6}&=3ab|\alpha_3|1+2ba|\beta_3|1-a|\alpha_3|c-c|\alpha_3|a-b|\beta_3|c-c|\beta_3|b+3|\alpha_3|ba+2|\beta_3|ab,
    \\
    q_{\alpha\beta_2}^{6}&=2bc|\beta_3|1+2ba|\alpha_2\gamma|1+3b|\alpha_3|c+3c|\alpha_3|b+3b|\alpha_2\gamma|b-2|\beta_3|(ba+ac)+2|\alpha_2\gamma|ab,
    \\
    q_{\omega_1\epsilon^!}^{6}&=2bac|\alpha_2\beta_2|a+aba|\alpha_2\beta_2|c-5bac|\alpha_4|a-3abc|\beta_4|b-2bac|\alpha_3\beta|b+8ab|\alpha_4|ba-6ab|\alpha_4|ac
    \\
    & \phantom{= \;}
    +6bc|\alpha_4|ba+3ba|\alpha_4|bc+3ac|\alpha_4|bc+3ab|\beta_4|ac+3bc|\beta_4|ac+10ba|\beta_4|ab-8ba|\beta_4|bc
    \\
    & \phantom{= \;}
    +8ac|\beta_4|ab-4ab|\gamma_4|ba-2ab|\alpha_2\beta_2|ba-5a|\alpha_4|bac-3b|\beta_4|abc-2b|\alpha_3\gamma|bac
    \\
    & \phantom{= \;}
    +2a|\alpha_2\beta_2|bac
    +c|\alpha_2\beta_2|aba- \omega_{1} a|\epsilon^!|a.
\end{align*}
\end{fact}

\begin{fact}
\label{fact:-X7}
Let $\rho=\rho^{7} : A\to A$ be the derivation of $A$ defined by $\rho^{7}(a)=aba-abc$, $\rho^{7}(ab)=\rho^{7}(ac)=abac$, $\rho^{7}(ba)=-abac$ 
and $\rho^{7}(x)=0$ for $x\in \mathcalboondox{B}\setminus \{a,ab,ba,ac\} $. 
Then the elements $q_{\omega_iu}=q_{\omega_iu}^{7}\in P^b_n$ in \eqref{eq:lifting} can be chosen as follows.
First, $q_{\beta_n}^{7}=q_{\gamma_n}^{7}=0$ for $n\in\NN$. 
Moreover, 
\begin{align*}
    q_{\alpha}^{7} &=ab|\alpha|1-ab|\gamma|1+a|\beta|a-a|\beta|c+1|\alpha|ba-1|\alpha|bc,
    \\
    q_{\alpha_2}^{7} &=ab|\alpha_2|1+ac|\alpha_2|1-a|\alpha\beta|a-1|\alpha_2|bc+2|\alpha_2|ba,
    \\
    q_{\alpha\beta}^{7}&=ba|\beta_2|1-ba|\gamma_2|1+(ba+ac)|\alpha\beta|1+a|\alpha_2|c+c|\beta_2|b+c|\beta_2|c-a|\gamma_2|c+b|\gamma_2|b-c|\gamma_2|b
    \\
    & \phantom{= \;}
    +c|\alpha\beta|a-c|\alpha\beta|c+a|\alpha\gamma|b+c|\alpha\gamma|b+1|\alpha_2|ba-1|\gamma_2|ba+1|\alpha\beta|(ba+ac),
    \\
    q_{\alpha\gamma}^{7}&=ab|\alpha_2|1-ab|\gamma_2|1+c|\alpha_2|a-c|\alpha_2|c+b|\beta_2|c+b|\alpha\beta|a-b|\alpha\beta|c+1|\beta_2|ab+1|\alpha\beta|ac,
    \\
    q_{\alpha_3}^{7}&=ab|\alpha_3|1+ac|\alpha_3|1-1|\alpha_3|bc+2|\alpha_3|ba,
    \\
    q_{\alpha_2\beta}^{7}&=ab|\alpha_3|1+2ba|\beta_3|1+2bc|\beta_3|1-ba|\gamma_3|1+2ba|\alpha_2\gamma|1-ab|\alpha\beta_2|1-a|\alpha_3|c+3b|\alpha_3|c
    \\
    & \phantom{= \;}
    +c|\alpha_3|b
    -a|\beta_3|a-a|\beta_3|b+b|\beta_3|a-b|\beta_3|c-2c|\beta_3|b+c|\beta_3|c-a|\gamma_3|b+a|\gamma_3|c-b|\gamma_3|b
    \\
    & \phantom{= \;}
    +2c|\gamma_3|a+c|\alpha_2\beta|c-a|\alpha_2\gamma|b+a|\alpha_2\gamma|c+b|\alpha_2\gamma|b-a|\alpha\beta_2|c+b|\alpha\beta_2|c-c|\alpha\beta_2|b
    \\
    & \phantom{= \;}
    +2|\alpha_3|ba
    +1|\beta_3|ab-3|\gamma_3|ab-1|\alpha_2\beta|bc+1|\alpha_2\beta|ba+2|\alpha_2\beta|ac,
    \\
    q_{\alpha_2\gamma}^{7}&=ac|\alpha_3|1-ba|\gamma_3|1-(ab+bc)|\alpha_2\beta|1+ba|\alpha\beta_2|1-b|\alpha_3|a+3b|\alpha_3|c+2c|\alpha_3|b+c|\alpha_3|c
    \\
    & \phantom{= \;}
    +a|\beta_3|c+c|\beta_3|a-a|\gamma_3|a+2a|\gamma_3|c-2b|\gamma_3|b+2c|\gamma_3|a+b|\alpha_2\beta|b-c|\alpha_2\beta|a+2c|\alpha_2\beta|c
    \\
    & \phantom{= \;}
    +a|\alpha_2\gamma|a+b|\alpha_2\gamma|b+b|\alpha_2\gamma|c+c|\alpha_2\gamma|a+b|\alpha\beta_2|a-2|\gamma_3|ab+2|\alpha_2\beta|ac,
    \\
    q_{\alpha\beta_2}^{7}&=ab|\alpha_3|1+ba|\beta_3|1-2ab|\gamma_3|1-a|\alpha_3|c-c|\alpha_3|a+2c|\alpha_3|c-a|\gamma_3|b+b|\gamma_3|c+c|\alpha_2\gamma|b
    \\
    & \phantom{= \;}
    +1|\alpha_3|ba
    +1|\beta_3|ab-1|\gamma_3|ba,
    \\
    q_{\omega_1\epsilon^!}^{7}&=2abc|\alpha_2\beta_2|a-3abc|\alpha_2\beta_2|c-aba|\alpha_2\beta_2|a+2aba|\alpha_2\beta_2|b-5bac|\alpha_4|b-7bac|\alpha_4|c
    \\
    & \phantom{= \;}
    +5abc|\beta_4|c
    +3aba|\gamma_4|a-3aba|\gamma_4|b+8ab|\alpha_4|ab+3ab|\alpha_4|bc+bc|\alpha_4|bc-2ba|\alpha_4|ba
    \\
    & \phantom{= \;}
    +4ac|\alpha_4|ba
    +6ac|\alpha_4|ac+9ab|\beta_4|ab-4bc|\beta_4|ab-5bc|\beta_4|bc-4ba|\beta_4|ba-7ba|\beta_4|ac
    \\
    & \phantom{= \;}
    +14ab|\gamma_4|ab+7bc|\gamma_4|ab-ba|\gamma_4|ba-ba|\gamma_4|ac-8ac|\gamma_4|ba+7ab|\alpha_3\beta|ac-3bc|\alpha_3\beta|ba
    \\
    & \phantom{= \;}
    +bc|\alpha_3\beta|ac+2ba|\alpha_3\beta|ab+3ba|\alpha_3\beta|bc+ba|\alpha_3\beta|ba+3ba|\alpha_3\beta|ac+3ac|\alpha_3\beta|bc
    \\
    & \phantom{= \;}
    -6ac|\alpha_3\beta|ba
    -3ac|\alpha_3\beta|ac-ab|\alpha_3\gamma|ab+5ab|\alpha_3\gamma|bc-9ab|\alpha_3\gamma|ba-5ab|\alpha_3\gamma|ac
    \\
    & \phantom{= \;}
    +3bc|\alpha_3\gamma|ab+6bc|\alpha_3\gamma|bc-2bc|\alpha_3\gamma|ba-2ba|\alpha_3\gamma|ab-2ba|\alpha_3\gamma|bc+6ac|\alpha_3\gamma|ab
    \\
    & \phantom{= \;}
    +7ab|\alpha_2\beta_2|ab+2ab|\alpha_2\beta_2|bc-ba|\alpha_2\beta_2|ba+2ba|\alpha_2\beta_2|ac+4ac|\alpha_2\beta_2|ac-2b|\alpha_4|bac
    \\
    & \phantom{= \;}
    +c|\alpha_4|bac+a|\beta_4|abc-2c|\beta_4|abc+4a|\gamma_4|aba+2b|\gamma_4|aba+2a|\alpha_2\beta_2|abc-2c|\alpha_2\beta_2|abc
    \\
    & \phantom{= \;}
    -a|\alpha_2\beta_2|aba-2b|\alpha_2\beta_2|bac-c|\alpha_2\beta_2|bac+\omega_1 3|\epsilon^!|(ba+ac-bc).
\end{align*}
\end{fact}

We will now apply the previous results to compute the Gerstenhaber brackets of $X_{i}$ for $i \in \llbracket 4,7 \rrbracket$ with all the other generators of the Hochschild cohomology of $A$.

\begin{prop}
\label{prop:X4}
The Gerstenhaber bracket $[X_i,X_j] \in \operatorname{HH}^{\bullet}(A)$ for $i \in \llbracket 4,7 \rrbracket$ and $j\in \llbracket 1,14 \rrbracket $ is given by 
\begin{equation}
\label{eq:X4 Xn}
[X_{i},X_{j}] = \begin{cases} 
    0, &\text{if $(i,j)\in \big(\llbracket 4,7 \rrbracket \times  \llbracket 1, 7 \rrbracket \big)\cup \big( \llbracket 4,6 \rrbracket \times \llbracket 9, 11 \rrbracket \big)$,}
    \\
    2 X_{i}, &\text{if $i \in \llbracket 4,7 \rrbracket $ and $j = 8$,}
    \\
    4 X_{1}X_{9}, &\text{if $i = 7$ and $j = 9$,}
    \\
     X_{1}X_{10}, &\text{if $i = 7$ and $j = 10$,}
    \\
    -X_{1} (X_9+X_{10}), &\text{if $i = 7$ and $j = 11$,}
    \\
    2 X_{1} X_{i+5}, &\text{if $i \in \llbracket 4,5 \rrbracket $ and $j = 12$,}
    \\
    2 X_{1} (X_{9} + X_{10}), &\text{if $i = 6$ and $j = 12$,}
    \\
    X_{1} X_9, &\text{if $i = 7$ and $j = 12$,}
    \\
    \tau_{i} 8 X_{4} X_{10}, &\text{if $i \in \llbracket 4,6 \rrbracket $ and $j= 13$,}
    \\
    4X_1X_{13}-4X_2X_{13}-8X_4X_{12}, &\text{if $i = 7$ and $j = 13$,}
    \\
    \tau_{i} \big((1/3)X_{i+5}^2-(4/3) X_{9} X_{10}\big), &\text{if $i \in \llbracket 4,6 \rrbracket $ and $j = 14$,}
    \\
    X_{9} X_{12}, &\text{if $i = 7$ and $j = 14$,}
\end{cases}
\end{equation}
where $\tau_{i} = 1$ if $i \in \llbracket 4, 5\rrbracket$ and $\tau_{6} = -1$. 
\end{prop}

\begin{proof}
Given $i \in \llbracket 4,7 \rrbracket$, let $\rho^{i}$ be the derivation of Fact \ref{fact:-X4}, Fact \ref{fact:-X5}, Fact \ref{fact:-X6}, and Fact \ref{fact:-X7},  respectively. 
Note that $G(\rho^{i})i_1=-Y_i$. 
By Theorem \ref{thm:HH1 HHn}, $[-Y_i, Y_j]$ is precisely the cohomology class of $\rho^{i} Y_j-Y_j \rho_n^{i}$ for $i \in \llbracket 4,7 \rrbracket$ and $j\in \llbracket 1,13 \rrbracket$, where $n$ is the cohomological degree of $Y_{j}$ and $\rho_n^{i}$ is obtained from \eqref{eq:lifting} together with Fact \ref{fact:-X4} for $i= 4$, Fact \ref{fact:-X5} for $i = 5$, Fact \ref{fact:-X6} for $i = 6$, and Fact \ref{fact:-X7} for $i = 7$. 
It is explicitly given by
\begin{equation}
\label{eq:-X_4}
    [-X_{4},X_{j}] = \begin{cases} 
    0, &\text{if $j \in \llbracket 1, 7 \rrbracket \cup \{ 9\}$,}
    \\
    -2 X_{4}, &\text{if $j = 8$,}
    \\
    \alpha\beta|(ab+bc)-\alpha\gamma|ac, &\text{if $j=10$,}
    \\
    -\alpha\beta|ab-\alpha\gamma|ba, &\text{if $j=11$,}
    \\
    \alpha_2|(bc-ba-ac) , &\text{if $j= 12$,}
    \\
    3\alpha_2\gamma|aba-5\alpha\beta_2|bac , &\text{if $j = 13$,}
\end{cases}
\end{equation}
and 
\begin{equation}
\label{eq:-X_5}
    [-X_{5},X_{j}] = \begin{cases} 
    0, &\text{if $j \in \llbracket 1, 7 \rrbracket  \cup \{ 10 \}$,}
    \\
    -2 X_{5}, &\text{if $j = 8$,}
    \\
    -\alpha\beta|bc+\alpha\gamma|(ba+ac), &\text{if $j=9$,}
    \\
    -\alpha\beta|ab-\alpha\gamma|ba, &\text{if $j=11$,}
    \\
    -\beta_2|(ab+bc -ac) , &\text{if $j= 12$,}
    \\
    -5\alpha_2\beta|abc+3\alpha_2\gamma|aba, &\text{if $j = 13$,}
\end{cases}
\end{equation}
as well as
\begin{equation}
\label{eq:-X_6}
    [-X_{6},X_{j}] = \begin{cases} 
    0, &\text{if $j \in \llbracket 1, 7 \rrbracket \cup \{ 11 \}$,}
    \\
    -2 X_{6}, &\text{if $j = 8$,}
    \\
    \alpha\beta|(bc-ab)-\alpha\gamma|(2ba+ac), &\text{if $j=9$,}
    \\
    -\alpha\beta|(ab+bc)+\alpha\gamma|ac, &\text{if $j=10$,}
    \\
    \gamma_2|(bc-ba-ac)-\alpha\beta|ba-\alpha\gamma|ab , &\text{if $j= 12$,}
    \\
    -10 \alpha_2\gamma|aba-2\alpha\beta_2|bac, &\text{if $j = 13$,}
\end{cases}
\end{equation}
together with 
\[   
[-X_7,X_{13}]=\alpha_3|(abc-2aba)+\alpha_2\beta|(2bac-6aba)-\alpha_2\gamma|(abc+4bac)+5\alpha\beta_2|(abc-aba),
\] 
and
\begin{equation}
\label{eq:-X_7}
    [-X_{7},X_{j}] = \begin{cases} 
    0, &\text{if $j \in \llbracket 1, 7 \rrbracket $,}
    \\
    -2 X_{7}, &\text{if $j = 8$,}
    \\
    \alpha_2|(bc-ab-ac-2ba)-\alpha\beta|(ba+ac)+\alpha\gamma|bc, &\text{if $j=9$,}
    \\
    \alpha\beta|ac-\alpha\gamma|(ab+bc), &\text{if $j=10$,}
    \\
    \alpha\beta|ba+\alpha\gamma|ab, &\text{if $j=11$,}
    \\
    (\alpha\beta+\alpha\gamma)|(bc-ba-ac), &\text{if $j= 12$.}
\end{cases}
\end{equation}

Next, we will compute 
$\varphi^{i}=[-Y_i, Y_{14}]= \rho^{i} Y_{14}-Y_{14}\rho_4^{i}$ for $i\in \llbracket 4,7 \rrbracket$. 
Using Fact \ref{fact:-X4}, it is easy to see that 
$\varphi^{4}(1|\beta_4|1)=\varphi^{4}(1|\gamma_4|1)=\varphi^{4}(\omega_1 1|\epsilon^!|1)=0$, whereas Fact \ref{fact:-X5} gives us immediately the identities $\varphi^{5}(1|\alpha_4|1)=\varphi^{5}(1|\gamma_4|1)=\varphi^{5}(\omega_1 1|\epsilon^!|1)=0$, Fact \ref{fact:-X6} tells us that 
$\varphi^{6}(1|\alpha_4|1)=\varphi^{6}(1|\beta_4|1)=\varphi^{6}(\omega_1 1|\epsilon^!|1)=0$, and Fact \ref{fact:-X7} yields that 
$\varphi^{7}(1|\beta_4|1)=\varphi^{7}(1|\gamma_4|1)=0$ and $\varphi^{7}(\omega_1 1|\epsilon^!|1)=3(bc-ba-ac)$. 
For $i \in \{4 ,7\}$ and $u\in  \mathcalboondox{B}^{!*}_{4}\setminus \{ \beta_4,\gamma_4 \} $ (resp., $i = 5$ and $u\in  \mathcalboondox{B}^{!*}_{4}\setminus \{ \alpha_4,\gamma_4 \} $, $i=6$ and $u\in  \mathcalboondox{B}^{!*}_{4}\setminus \{ \alpha_4,\beta_4 \} $), we have that 
\begin{equation}
    \begin{split}
        \varphi^{i}(1|u|1)
         = (\rho^{i} Y_{14}-Y_{14}\rho_4^{i})(1|u|1)
         =-Y_{14}(q_{u}^{i})
         =-\lambda^u_{i},
    \end{split}
\end{equation}
where $\lambda^u_{i}\in\Bbbk$ is the coefficient of $\omega_11|\epsilon^!|1$ in $q_{u}^{i}$.
It is easy to check that 
\begin{align*}
    \rho_3^{4}\delta^b_4(1|\alpha_4|1)&=bac|\alpha_3|1+abc|\beta_3|1+aba|\alpha_2\gamma|1+v_{\alpha_4}^{4}, 
    \\
    \rho_3^{4}\delta^b_4(1|\alpha_3\beta|1)&=-2aba|\alpha_3|1+3abc|\gamma_3|1+bac|\alpha_2\beta|1+abc|\alpha_2\gamma|1+2aba|\alpha\beta_2|1+v_{\alpha_3\beta}^{4}, 
    \\
    \rho_3^{4}\delta^b_4(1|\alpha_3\gamma|1)&=-2aba|\beta_3|1+2bac|\gamma_3|1+2aba|\alpha_2\beta|1+bac|\alpha_2\gamma|1+abc|\alpha\beta_2|1+v_{\alpha_3\gamma}^{4},
    \\
    \rho_3^{4}\delta^b_4(1|\alpha_2\beta_2|1)&=-2bac|\alpha_3|1-2abc|\beta_3|1+4aba|\gamma_3|1+2abc|\alpha_2\beta|1+2bac|\alpha\beta_2|1+v_{\alpha_2\beta_2}^{4},
\end{align*}
and 
\begin{align*}
    \rho_3^{5}\delta^b_4(1|\beta_4|1)&=bac|\alpha_3|1+abc|\beta_3|1+aba|\alpha_2\gamma|1+v_{\beta_4}^{5}, 
    \\
    \rho_3^{5}\delta^b_4(1|\alpha_3\beta|1)&=-2aba|\alpha_3|1+2abc|\gamma_3|1+bac|\alpha_2\beta|1+aba|\alpha\beta_2|1+v_{\alpha_3\beta}^{5}, 
    \\
    \rho_3^{5}\delta^b_4(1|\alpha_3\gamma|1)&=abc|\alpha_3|1-2aba|\beta_3|1+2bac|\gamma_3|1+2aba|\alpha_2\beta|1+bac|\alpha_2\gamma|1+abc|\alpha\beta_2|1+v_{\alpha_3\gamma}^{5},
    \\
    \rho_3^{5}\delta^b_4(1|\alpha_2\beta_2|1)&=-3bac|\alpha_3|1-2abc|\beta_3|1+3aba|\gamma_3|1+2abc|\alpha_2\beta|1-aba|\alpha_2\gamma|1+bac|\alpha\beta_2|1
    \\
    & \phantom{= \;}
    +v_{\alpha_2\beta_2}^{5},
\end{align*}
as well as 
\begin{align*}
    \rho_3^{6}\delta^b_4(1|\gamma_4|1)&=-abc|\beta_3|1+aba|\gamma_3|1+bac|\alpha\beta_2|1+v_{\gamma_4}^{6}, 
    \\
    \rho_3^{6}\delta^b_4(1|\alpha_3\beta|1)&=4aba|\alpha_3|1-2bac|\beta_3|1+2abc|\alpha_2\gamma|1+2aba|\alpha\beta_2|1+v_{\alpha_3\beta}^{6}, 
    \\
    \rho_3^{6}\delta^b_4(1|\alpha_3\gamma|1)&=-abc|\alpha_3|1+2aba|\beta_3|1+aba|\alpha_2\beta|1+bac|\alpha_2\gamma|1+v_{\alpha_3\gamma}^{6},
    \\
    \rho_3^{6}\delta^b_4(1|\alpha_2\beta_2|1)&=4bac|\alpha_3|1+4abc|\beta_3|1+4aba|\alpha_2\gamma|1
    +v_{\alpha_2\beta_2}^{6},
\end{align*}
together with 
\begin{align*}
    \rho_3^{7}\delta^b_4(1|\alpha_4|1)&=(aba-abc)|\alpha_3|1+v_{\alpha_4}^{7}, 
    \\
    \rho_3^{7}\delta^b_4(1|\alpha_3\beta|1)&=-abc|\alpha_3|1+3bac|\alpha_3|1+3aba|\beta_3|1+2abc|\beta_3|1-aba|\gamma_3|1-2bac|\gamma_3|1
    \\
    & \phantom{= \; }
    -abc|\alpha_2\beta|1
    +2aba|\alpha_2\gamma|1 +v_{\alpha_3\beta}^{7}, 
    \\
    \rho_3^{7}\delta^b_4(1|\alpha_3\gamma|1)&=2aba|\alpha_3|1+2bac|\alpha_3|1+2abc|\beta_3|1-2bac|\beta_3|1-2aba|\gamma_3|1-2abc|\gamma_3|1
    \\
    & \phantom{= \; }
    -abc|\alpha_2\beta|1
    +aba|\alpha_2\gamma|1+abc|\alpha_2\gamma|1+aba|\alpha\beta_2|1-bac|\alpha\beta_2|1 +v_{\alpha_3\gamma}^{7},
    \\
    \rho_3^{7}\delta^b_4(1|\alpha_2\beta_2|1)&=aba|\alpha_3|1-abc|\alpha_3|1+aba|\beta_3|1-abc|\gamma_3|1
    +v_{\alpha_2\beta_2}^{7},
\end{align*}
where $v_u^{i}\in \oplus_{j\in \llbracket 0,2 \rrbracket } (A_j\otimes (A^!_{-3})^{*}\otimes A_{3-j})  $ if $i \in  \{4, 7\}$ and $u\in  \mathcalboondox{B}^{!*}_{4}\setminus \{ \beta_4,\gamma_4 \}$, or if $i = 5$ and $u\in  \mathcalboondox{B}^{!*}_{4}\setminus \{ \alpha_4,\gamma_4 \}$, or if $i= 6$ and $u\in  \mathcalboondox{B}^{!*}_{4}\setminus \{ \alpha_4,\beta_4 \}$. 

Since, $q_{u}^{i}$ is of the form $q_{u}^{i}=B_u^{i}+\lambda^u_{i}\omega_11|\epsilon^!|1$ by degree reasons, where $B_u^{i}\in K^b_4$, 
and we have by definition that $\delta^b_4(q_u^{i})=\rho_3^{i}\delta^b_4(1|u|1)$, 
we see that  
\begin{equation}
\label{eq:B lambda -X4}
    \begin{split}
        d^b_4(B_u^{i})=\delta^b_4(q_{u}^{i})-\lambda^u_{i} f^b_0(1|\epsilon^!|1)=\rho_3^{i}\delta^b_4(1|u|1)-\lambda^u_{i} f^b_0(1|\epsilon^!|1) 
    \end{split}
\end{equation}
for $i \in \llbracket 4, 7 \rrbracket$. 
Using the explicit expression of the differential $d^b_4$ given in \cite{es zl}, Fact 3.1, 
it is clear that the coefficients of $aba|\gamma_3|1$ and $aba|\alpha_2\gamma|1$ in $d^b_4(B)$ coincide for all $B \in  K^b_4$. 
Comparing the coefficients of $aba|\gamma_3|1$ and $aba|\alpha_2\gamma|1$ in both sides of the equation \eqref{eq:B lambda -X4}, 
together with the expression of $f^b_0(1|\epsilon^!|1)$ given in \cite{es zl}, (3.2),
we get
\begin{equation}
\label{eq:lambda -X4}
    \begin{split}
        \lambda^{\alpha_4}_{4}=\lambda^{\beta_4}_{5} &= - \lambda^{\gamma_4}_{6} = 1/3, \hskip 2mm \lambda^{\alpha_3\beta}_{i}=\lambda^{\alpha_3\gamma}_{i}=0, \hskip 2mm \lambda^{\alpha_2\beta_2}_{i}= - \tau_{i} 4/3,
        \\
        &\lambda^{\alpha_{4}}_{7} = \lambda^{\alpha_{2} \beta_{2}}_7 = 0, \hskip 2mm \text{ and } \hskip 2mm \lambda^{\alpha_{3} \beta}_{7} = \lambda^{\alpha_{3} \gamma}_7 = 1,
    \end{split}
\end{equation} 
for $i \in \llbracket 4, 6 \rrbracket$, where $\tau_{i} = 1$ if $i \in \llbracket 4, 5 \rrbracket$ and $\tau_{6} = -1$. 
Hence, we obtain that 
\begin{equation}
\label{eq:lambda -X4bis}
    \begin{split}
&\varphi^{4}(1|\alpha_4|1)=-1/3, \hskip 2mm 
\varphi^{4}(1|\alpha_3\beta|1)=\varphi^{4}(1|\alpha_3\gamma|1)=0 \hskip 1mm \text{  
and } \hskip 1mm
\varphi^{4}(1|\alpha_2\beta_2|1)=4/3,
\\
&\varphi^{5}(1|\beta_4|1)=-1/3, \hskip 2mm
\varphi^{5}(1|\alpha_3\beta|1)=\varphi^{5}(1|\alpha_3\gamma|1)=0 
\hskip 1mm \text{  
and } \hskip 1mm 
\varphi^{5}(1|\alpha_2\beta_2|1)=4/3,
\\
&\varphi^{6}(1|\gamma_4|1)=1/3, \hskip 5mm
\varphi^{6}(1|\alpha_3\beta|1)=\varphi^{6}(1|\alpha_3\gamma|1)=0 \hskip 1mm \text{  
and } \hskip 1mm \varphi^{6}(1|\alpha_2\beta_2|1)=-4/3, 
\\
&\varphi^{7}(1|\alpha_4|1)=\varphi^{7}(1|\alpha_2\beta_2|1)=0 \hskip 5mm \text{ and } \hskip 5mm \varphi^{7}(1|\alpha_3\beta|1)=\varphi^{7}(1|\alpha_3\gamma|1)=-1.
    \end{split}
\end{equation} 
In consequence, we get 
\begin{equation}
\label{eq:lambda -X4bisbis}
    \begin{split}
    [-X_4, X_{14}]&=(4/3)\alpha_2\beta_2|1-(1/3)\alpha_4|1, 
    \\
    [-X_5, X_{14}]&=(4/3)\alpha_2\beta_2|1-(1/3)\beta_4|1, 
    \\
    [-X_6, X_{14}]&=(1/3)\gamma_4|1-(4/3)\alpha_2\beta_2|1,
    \\
    [-X_7, X_{14}]&=-(\alpha_3\beta+\alpha_3\gamma)|1+3\omega^*_1\epsilon^!|(bc-ba-ac)
\end{split}
\end{equation} 

Using the coboundaries 
$g^2_{j,2}\in \tilde{\mathfrak{B}}^2_2$ for $j\in \llbracket 4,6\rrbracket$ and $e^3_{k,3}\in \tilde{\mathfrak{B}}^3_3$ for $k\in \llbracket 7,8\rrbracket$
given in \cite{es zl}, Subsubsection 5.3.1,  \eqref{eq:some cup products1} as well as the identities 
\begin{equation}
\label{eq:some cup products 3}
\begin{split}
    \alpha_2\beta|abc=X_4X_{10}, \hskip 2mm
    \alpha_4|1=X_9^2, \hskip 2mm \text{ and } \hskip 2mm
    \alpha_2\beta_2|1=X_9X_{10},
    \end{split}
\end{equation}
which follow from \cite{es zl}, Fact 6.3 and (6.2), we can rewrite several brackets as 
\begin{align*}
    [-X_4, X_{10}]&=\alpha\beta|(ab+bc)-\alpha\gamma|ac-g^2_{5,2}=0, 
    \\
    [-X_4,X_{11}]&=-\alpha\beta|ab-\alpha\gamma|ba+g^2_{4,2}=0,
    \\
    [-X_4,X_{12}]&=\alpha_2|(bc-ba-ac)-g^2_{6,2}=-2\alpha_2|(ab+ba)=-2X_1X_9, 
    \\
    [-X_4,X_{13}]&=3\alpha_2\gamma|aba-5\alpha\beta_2|bac -5e^3_{7,3}-3e^3_{8,3}=-8\alpha_2\beta|abc=-8X_4X_{10}, 
    \\
    [-X_4,X_{14}]&=(4/3)\alpha_2\beta_2|1-(1/3)\alpha_4|1=(4/3)X_9X_{10}-(1/3)X_9^2.
\end{align*}
Analogously, using the coboundaries 
$g^2_{j,2}\in \tilde{\mathfrak{B}}^2_2$ for $j \in \{ 4,5,7 \}$ and $e^3_{8,3}\in \tilde{\mathfrak{B}}^3_3$
given in \cite{es zl}, Subsubsection 5.3.1, \eqref{eq:some cup products1}, \eqref{eq:some cup products 3} and the identity 
    $\beta_4|1=X^2_{10}$
given in \cite{es zl}, Fact 6.3, 
we get that 
\begin{align*}
    [-X_5, X_{9}]&=-\alpha\beta|bc+\alpha\gamma|(ba+ac)-g^2_{4,2}+g^2_{5,2}=0, 
    \\
    [-X_5,X_{11}]&=-\alpha\beta|ab-\alpha\gamma|ba+g^2_{4,2}=0,
    \\
    [-X_5,X_{12}]&=-\beta_2|(ab+bc-ac) -g^2_{7,2}=-2\beta_2|(ab+ba)=-2X_1X_{10},
    \\
    [-X_5,X_{13}]&=-5\alpha_2\beta|abc+3\alpha_2\gamma|aba -3e^3_{8,3}=-8\alpha_2\beta|abc=-8X_4X_{10}, 
    \\
    [-X_5,X_{14}]&=(4/3)\alpha_2\beta_2|1-(1/3)\beta_4|1=(4/3)X_9X_{10}-(1/3)X_{10}^2.
\end{align*} 
Moreover, 
using the coboundaries 
$g^2_{j,2}\in \tilde{\mathfrak{B}}^2_2$ for $j
\in \llbracket 1,5 \rrbracket$ and $e^3_{k,3}\in \tilde{\mathfrak{B}}^3_3$ for $k
\in \llbracket 7,8 \rrbracket$
given in \cite{es zl}, Subsubsection 5.3.1, \eqref{eq:some cup products1}, \eqref{eq:some cup products 3} and the identity
    $\gamma_4|1=X^2_{11}$
given in \cite{es zl}, Fact 6.3,
we obtain
\begin{align*}
    [-X_6, X_{9}]&=\alpha\beta|(bc-ab)-\alpha\gamma|(2ba+ac)+2g^2_{4,2}-g^2_{5,2}=0,
    \\
    [-X_6,X_{10}]&=-\alpha\beta|(ab+bc)+\alpha\gamma|ac+g^2_{5,2}=0,
    \\
    [-X_6,X_{12}]&=\gamma_2|(bc-ba-ac)-\alpha\beta|ba-\alpha\gamma|ab -2g^2_{1,2}-2g^2_{2,2} -g^2_{3,2}=-2(\alpha_2+\beta_2)|(ab+ba)
    \\ 
    &=-2X_1(X_9+X_{10}), 
    \\
    [-X_6,X_{13}]&=-10 \alpha_2\gamma|aba-2\alpha\beta_2|bac-2e^3_{7,3} +10e^3_{8,3}=8\alpha_2\beta|abc=8X_4X_{10}, 
    \\
    [-X_6,X_{14}]&=(1/3)\gamma_4|1-(4/3)\alpha_2\beta_2|1=(1/3)X_{11}^2-(4/3)X_9X_{10}.
\end{align*}
Finally, using the coboundaries 
$g^2_{j,2}\in \tilde{\mathfrak{B}}^2_2$ for $j
\in \llbracket 1,6 \rrbracket \setminus\{3\} $ and  
$e^3_{k,3}\in \tilde{\mathfrak{B}}^3_3$ for $k
\in \llbracket 1,4 \rrbracket \cup \llbracket 9,10 \rrbracket $
given in \cite{es zl}, Subsubsection 5.3.1, \eqref{eq:some cup products1} and 
\begin{equation}
\label{eq:some cup products X7}
\begin{split}
   & \alpha_3|(aba-abc)  =X_7X_9, \quad 
    \alpha_3|aba+\beta_3|bac =X_7(X_9+X_{10})-2X_6X_{12}, 
    \\ 
   & (\alpha_3+\beta_3)|aba =X_{6}X_{12},\quad 
   (\alpha_3\beta+\alpha_3\gamma)|1+3\omega_1\epsilon^!|(ba-bc+ac)=X_9X_{12}, 
    \end{split}
\end{equation}
given in \cite{es zl}, Fact 6.3, or in \cite{es zl}, (6.2), 
together with the second element in the fifth and the eighth line, the first element in the ninth line of \cite{es zl} (6.5), 
we have that
\begin{align*}
    [-X_7, X_{9}]&=\alpha_2|(bc-ab-ac-2ba)-\alpha\beta|(ba+ac)+\alpha\gamma|bc-g^2_{1,2}-g^2_{6,2}=-4\alpha_2|(ab+ba)
    \\
    &=-4X_1X_9,
    \\
    [-X_7,X_{10}]&=\alpha\beta|ac-\alpha\gamma|(ab+bc)-g^2_{2,2}=-\beta_2|(ab+ba)=-X_1X_{10}, 
    \\
    [-X_7,X_{11}]&=\alpha\beta|ba+\alpha\gamma|ab+g^2_{1,2}+g^2_{2,2}=(\alpha_2+\beta_2)|(ab+ba)=X_1(X_9+X_{10}),
    \\
    [-X_7,X_{12}]&=(\alpha\beta+\alpha\gamma)|(bc-ba-ac)-g^2_{1,2}+g^2_{4,2}-g^2_{5,2}=-\alpha_2|(ab+ba)=-X_1X_9,
    \\ 
    [-X_7,X_{13}]&=\alpha_3|(abc-2aba)+\alpha_2\beta|(2bac-6aba)-\alpha_2\gamma|(abc+4bac)+5\alpha\beta_2|(abc-aba)
    \\
    & \phantom{= \; }
    -(1/3)(23e^3_{1,3}+11e^3_{2,3}-32e^3_{3,3}-16e^3_{4,3}-5e^3_{9,3}+6e^3_{10,3})
    \\
    &=(8/3)\alpha_3|(aba-abc)-(32/3)(\alpha_3+\beta_3)|aba-(16/3)(\alpha_3|aba+\beta_3|bac)
    \\
    &=-(8/3)X_7(X_9+2X_{10})=-4X_1X_{13}+4X_2X_{13}+8X_4X_{12}, 
    \\
    [-X_7,X_{14}]&=-(\alpha_3\beta+\alpha_3\gamma)|1+3\omega^*_1\epsilon^!|(bc-ba-ac)=-X_9X_{12}. 
\end{align*}
The proposition is thus proved.
\end{proof}

\begin{rk} 
\label{remark:grad-vanishing-1}
Note that vanishing of $[X_{i},X_{j}]$ for $i \in \llbracket 4 , 7 \rrbracket$ and $j \in \llbracket 3 , 7 \rrbracket$ in Proposition \ref{prop:X4} 
also follows from a simple degree argument based on Corollary \ref{cor:gr Hom} and the Hilbert series of the Hochschild cohomology given in \cite{es zl}, Cor. 5.9. 
\end{rk}


\subsection{\texorpdfstring{Gerstenhaber brackets}{Gerstenhaber brackets}}
\label{sub:Gerstenhaber brackets} 

We will finally compute the remaining Gerstenhaber brackets. 
We start with the following result, which is a sort of descending argument. 

\begin{lem}
\label{lemma:tech}
Let $H = \oplus_{n \in \NN_{0}} H^{n}$ be a Gerstenhaber algebra with bracket $[\hskip 0.6mm , ]$. 
Let $x \in H^{n+1}$, $y \in H^{n}$, $a_{x} \in H^{0}$, $a_{y} \in H^{1}$ and $z \in H^{m}$ satisfy that $a_x x = a_{y} y$, and there is a vector subspace $M \subseteq H^{n+m-1}$ such that $[y,z] \in M$ and the map $\mu_{a_{y}} : M \rightarrow H^{n+m}$ sending $v\in M$ to $a_{y} v$ is injective. 
Then, $[y,z]$ is the unique element $v \in M$ satisfying that $a_{y} v$ coincides with
 \begin{equation}
\label{eq:tech0}
        (-1)^{m-1} \big( a_x[x,z] + [a_{x},z]x - [a_{y},z] y\big).
\end{equation}  
\end{lem}
\begin{proof} 
By \eqref{property 2} we get that 
\begin{equation}
\label{eq:proof tech}
      [a_{x}x,z]=[a_{x},z]x + a_{x}[x,z]
\text{ and } [a_{y}y,z]=[a_{y},z]y+(-1)^{m-1}a_{y}[y,z].
\end{equation}
These identities together with $a_xx=a_{y}y$ imply
\begin{equation}
\label{eq:tech}
      a_{y}[y,z] = (-1)^{m-1} \big( a_x[x,z] + [a_{x},z]x - [a_{y},z] y\big).
\end{equation} 
Hence, the right member  
is in the image of the injective map $\mu_{a_{y}}$, and the result follows.
\end{proof}

\begin{rk}
We will apply the previous lemma to the case when $H = \operatorname{HH}^{\bullet}(A)$ is the Hochschild cohomology of a graded algebra $A$, so $H$ is endowed with an extra grading, called internal (see Corollary \ref{cor:gr Hom}), the elements $x,y,z, a_{x}, a_{y}$ are homogeneous for both gradings and $M \subseteq H^{n+m-1}$ is the subspace of internal degree equal to the sum of those of $y$ and $z$. 
In this case, the methods given in Subsections \ref{subsection:Gerstenhaber brackets of HH^0} and \ref{subsection:Gerstenhaber brackets of HH^1} allow to compute the last two brackets of \eqref{eq:tech0}, whereas the first one will usually vanish by degree reasons. 
\end{rk}

\begin{prop}
\label{prop:remaining brackets} 
Let $A =\FK(3)$ be the Fomin-Kirillov algebra on $3$ generators. 
Then, we have the Gerstenhaber brackets $[X_i,X_j]=0 $ for $i,j\in\llbracket 9,14 \rrbracket \setminus \{13\}$ 
and 
\begin{equation}
\label{eq:X_13 X_n}
    [X_{13},X_{j}] = \begin{cases} 
    2 X^2_{j}, &\text{if $j\in \llbracket 9,11\rrbracket $,}
    \\
    -6X_1X_{14}+6X_2X_{14}+2X_9X_{12}, &\text{if $j = 12$,}
    \\
    0, &\text{if $j = 13$,}
    \\
    4(X_9+X_{10}+X_{11})X_{14}, &\text{if $j =14$.}
\end{cases}
\end{equation}
\end{prop}
\begin{proof}
Recall that, by Corollary \ref{cor:gr Hom}, 
the Gerstenhaber bracket satisfies that $[ \hskip 0.6mm , ]:H^{n_1}_{m_1}\times H^{n_2}_{m_2}\to H^{n_1+n_2-1}_{m_1+m_2-1}$, 
where $H^{n_i}_{m_i}$ has internal degree $m_i-n_i$ for $i=1,2$. 
Using this degree argument together with the Hilbert series of the Hochschild cohomology computed in \cite{es zl}, Cor. 5.9, we easily see that
$[X_i,X_j]=0$ for $i,j\in\llbracket 9,14 \rrbracket \setminus \{13\}$.  
Moreover, $[X_{13},X_{13}]=0$ by \eqref{property 1}. 

It remains to compute $[X_{13},X_{j}]$ for all $j \in \llbracket 9 , 14 \rrbracket \setminus \{ 13 \}$. 
Note first the identities 
\begin{equation} 
\label{eq:auxx}
\begin{split}
[X_{8},X_9] X_{13} - 6 [X_{3},X_{9}] X_{14} &=2X_9X_{13}+12X_4X_{14}=2X_8X_{9}^2,
\\ 
[X_{8},X_{10}] X_{13} - 6 [X_{3},X_{10}] X_{14}&=2X_{10}X_{13}+12X_5X_{14}=2X_8X_{10}^2,
\\ 
[X_{8},X_{11}] X_{13} - 6 [X_{3},X_{11}] X_{14} &=2X_{11}X_{13}-12X_6X_{14}=2X_8X_{11}^2, 
\\ 
[X_{8},X_{12}] X_{13} - 6 [X_{3}, X_{12}] X_{14} &=2X_{12}X_{13}-12X_7X_{14}+6X_1X_8X_{14}-6X_2X_8X_{14}
\\
&=2X_8X_{11}X_{12}+6X_1X_8X_{14}
=2X_8X_{12}^2-4X_8X_9X_{10}
\\
&=2X_8X_9X_{12}-6X_1X_8X_{14}+6X_2X_8X_{14},
\\ 
[X_{8},X_{14}] X_{13} - 6 [X_{3}, X_{14}] X_{14}&=4X_8(X_9+X_{10}+X_{11})X_{14},
\end{split}
\end{equation}
where the first equality of the first fourth lines as well as that of the last line follows from Propositions \ref{prop:bracket-H0-Xi} and \ref{prop:X_8}, 
and we used the first element of the seventh and the eighth line of \cite{es zl}, (6.5), as well as its last four elements, for the remaining equalities. 
The penultimate element of the ninth line of \cite{es zl}, (6.5), also tells us that $6X_3 X_{14} = X_8 X_{13} \in \HH^{\bullet}(A)$. 

Notice now that, by degree reasons, $[X_{13},X_j]\in H^4_0$ for $j\in\llbracket 9,12\rrbracket$ and $H^4_0$ is precisely the subspace of $\operatorname{HH}^{4}(A)$ spanned by the elements 
$X_9^2$, 
$X_{10}^2$, 
$X_{11}^2$, 
$ X_9X_{12}-3X_1X_{14}+3X_2X_{14}$, 
$X_9X_{10}$, 
$X_1X_{14}$
and  
$X_2X_{14}$. 
On the other hand, $[X_{13},X_{14}]\in H^6_{-2}=\omega^*_1 H^2_0$, by degree reasons, and $\omega^*_1 H^2_0$ is the subspace of $\operatorname{HH}^{4}(A)$ spanned by $X_9X_{14}$, $X_{10}X_{14}$, $X_{11}X_{14}$ and $X_{12}X_{14}$.
Let us denote by ${}^{j}M \subseteq \HH^{4}(A)$ the subspace given by $H^4_0$ if $j \in \llbracket 9 , 12 \rrbracket$ and by $H^{6}_{-2}$ if $j = 14$. 
Since the elements 
$X_8X_9^2$, 
$X_8X_{10}^2$, 
$X_8X_{11}^2$, 
$X_8X_9X_{12}-3X_1X_8X_{14}+3X_2X_8X_{14}$, 
$X_8X_9X_{10}$, 
$X_1X_8X_{14}$, 
and 
$X_2X_8X_{14}$
are linearly independent, by the second equalities of the first four lines of \eqref{eq:auxx} together with \cite{es zl}, (6.7) and (6.8), 
the map ${}^{j}M \rightarrow \HH^{5}(A)$ given by left multiplication by $X_{8}$ is injective 
for $j \in \llbracket 9 , 12 \rrbracket$. 
Similarly, the elements $X_8X_9X_{14}$, $X_8X_{10}X_{14}$, $X_8X_{11}X_{14}$ and $X_8X_{12}X_{14}$ are linearly independent, by \cite{es zl}, (6.8), 
so the map ${}^{14}M \rightarrow \HH^{7}(A)$ given by left multiplication by $X_{8}$ is also injective. 

Finally, applying Lemma \ref{lemma:tech} to $x = X_{14}$, $y = X_{13}$, $z = X_{j}$, $a_{x} = 6 X_{3}$, $a_{y} = X_{8}$ and $M = {}^{j}M$ for $j\in\llbracket 9,14\rrbracket \setminus \{ 13 \}$, together with the fact remarked at the beginning of the proof that $[X_{14},X_{j}] = 0$ and \eqref{eq:auxx}, the result follows. 
\end{proof}

We can summarize the calculations of the Gerstenhaber brackets on $\HH^{\bullet}(A)$ done in Propositions \ref{prop:bracket-H0-Xi}, \ref{prop:X_8}, \ref{prop:X4} 
and \ref{prop:remaining brackets} in the following table, where the brackets strictly below the diagonal are not displayed since they can be obtained using Lemma \ref{property of G}. 

\begin{table}[H]
	\begin{center}
        \resizebox{\textwidth}{24mm}{
       \begin{tabular}{|c|cccccccccccccc|}
		\hline
		 \diagbox[width=13mm,height=5.4mm]{$\rho$}{$\phi$}  & $X_1$ & $X_2$  & $X_3$ & $X_4$ & $X_5$ & $X_6$ & $X_7$ & $X_8$ & $X_9$ & $X_{10}$ & $X_{11}$ & $X_{12}$ & $X_{13}$ & $X_{14}$
        \\
        \hline
        $X_1$ & $0$ & $0$ & $0$ & $0$ & $0$ & $0$ & $0$ & $2X_1$ & $0$ & $0$ & $0$ & $0$ & $4X_1(X_9+X_{10})$ &  $0$
        \\
        $X_2$ &  & $0$ & $0$ & $0$ & $0$ & $0$ & $0$ & $2X_2$ & $0$ & $0$  & $0$ & $0$ & $4X_1X_{10}$ & $0$
        \\
		$X_3$ &  &  & $0$ & $0$ & $0$ & $0$ & $0$ & $4X_3$ & $-2X_4$ & $-2X_5$  & $2X_6$ & $2X_7-X_1X_8+X_2X_8$ & $4X_3(X_9+X_{10}+X_{11})$ &  $X_{13}-(2/3)X_8(X_9+X_{10}+X_{11}) $
        \\
		$X_4$ &  &  &  & $0$ & $0$& $0$& $0$& $2X_4$& $0$& $0$& $0$& $2X_1X_9$& $8X_4X_{10}$&  $(1/3)X_9^2-(4/3)X_9X_{10}$
        \\
        $X_5$ &  &  & &  & $0$ & $0$ & $0$ & $2X_5$ & $0$ & $0$ & $0$ & $2X_1X_{10}$ &$8X_4X_{10}$ & $(1/3)X_{10}^2-(4/3)X_9X_{10}$ 
        \\
        $X_6$ &  &  &  &  &  & $0$ & $0$ & $2X_6$ & $0$ & $0$ & $0$ & $2X_1(X_9+X_{10})$ & $-8X_4X_{10}$ & $(4/3)X_9X_{10}-(1/3)X_{11}^2$ 
        \\
        $X_7$ &  &  &  &  &  &  & $0$ & $2X_7$ & $4X_1X_9$ & $X_1X_{10}$ & $-X_1(X_9+X_{10})$ & $X_1X_9$ & $4X_1X_{13}-4X_2X_{13}-8X_4X_{12}$ & $X_9X_{12}$ 
        \\
        $X_8$ &  &  &  &  &  &  &  & $0$ & $2X_9$ & $2X_{10}$ & $2X_{11}$ & $2X_{12}$ & $2X_{13}$ & $6X_{14}$
        \\
        $X_9$ &  &  & &  &  &  &  &  & $0$ & $0$ & $0$ & $0$ & $-2X_9^2$  & $0$
        \\
        $X_{10}$ &  &  & &  &  &  &  &  &  & $0$ & $0$ & $0$ & $-2X_{10}^2$  & $0$
        \\
        $X_{11}$ &  &  & &  &  &  &  &  &  &  & $0$ & $0$ & $-2X_{11}^2$ & $0$
        \\
        $X_{12}$ &  &  & &  &  &  &  &  &  &  &  & $0$ & $6X_1X_{14}-6X_2X_{14}-2X_9X_{12}$   & $0$
        \\
        $X_{13}$ &  &  &  &   &  &  &  &  &  &  &  &  & $0$ &  $4(X_9+X_{10}+X_{11})X_{14}$
        \\
        $X_{14}$ &  &  & &  &  &  &  &  &  &  &  &  &  & $0$
        \\
        \hline
	\end{tabular}
    }
	\end{center}
	\caption{Gerstenhaber brackets $[\rho,\phi]$.}	
	\label{table:Gerstenhaber brackets}
\end{table}

\begin{prop}
\label{prop:no generator}
There is no generator of the Gerstenhaber bracket on the Hochschild cohomology $\HH^{\bullet}(A)$ of $A=\FK(3)$, \textit{i.e.} there is no map $\Delta:\HH^{\bullet}(A)\to \HH^{\bullet}(A)$ of degree $-1$ such that 
\begin{equation}
    \label{eq:Batalin-Vilkovisky}
\begin{split}
[x,y]=(-1)^{|x|}\big(\Delta(xy)-\Delta(x)y-(-1)^{|x|}x\Delta (y)\big)
\end{split}
\end{equation}
for all homogeneous elements $x,y\in\HH^{\bullet}(A)$, where $|x|$ is the cohomological degree of $x$. 
In particular, there is no Batalin-Vilkovisky structure on $\HH^{\bullet}(A)$ inducing the Gerstenhaber bracket.
\end{prop}
\begin{proof}
Assume that \eqref{eq:Batalin-Vilkovisky} holds. 
Obviously, $\Delta (\HH^0(A))=0$. 
Applying the results in Table \ref{table:Gerstenhaber brackets} and \eqref{eq:Batalin-Vilkovisky}, 
we get 
$-4X_3=[X_8,X_3]=\Delta (X_8)X_3$, 
and 
$0=[X_i,X_j]=\Delta(X_i)X_j$ for $i\in \llbracket 4,7 \rrbracket $ and $j \in  \llbracket 1,3 \rrbracket$, since $X_{8} X_{3} = X_i X_j = 0$ in that case (see the first two lines of \cite{es zl}, (6.5)). 
Hence, 
$\Delta (X_8)\in -4+\operatorname{span}_{\Bbbk}\langle X_1,X_2,X_3\rangle$ and $\Delta(X_i)\in \operatorname{span}_{\Bbbk}\langle X_1,X_2,X_3\rangle $ for $i\in  \llbracket 4,7 \rrbracket $, where $\operatorname{span}_{\Bbbk}\langle X_1,X_2,X_3\rangle $ is the $\Bbbk$-subspace spanned by $X_1,X_2,X_3$.  
Moreover, 
\begin{equation}
\label{eq:eqq}
\begin{split}
    -2X_4 & =[X_3,X_9]=\Delta (X_3X_9)-X_3\Delta (X_9)=\Delta (X_3X_9), 
    \\
2X_4 & =[X_4,X_8]=-\Delta(X_4X_8)+\Delta(X_4)X_8-X_4\Delta (X_8)=-\Delta(X_4X_8)+\Delta(X_4)X_8+4X_4,
\end{split}
\end{equation}  
where we used that $X_{4} X_{i} = X_{3} X_{k} = 0$ for $i \in \llbracket 1 , 3 \rrbracket$ and $k \in \llbracket 4,8 \rrbracket$, by the first two lines of \cite{es zl}, (6.5). 
Since $X_3X_9=X_4X_8\in \HH^2(A)$ (see the penultimate element of the third line of \cite{es zl}, (6.5)), adding the equations \eqref{eq:eqq}, we obtain $\Delta(X_4)X_8+4X_4=0$. 
The identity $\Delta(X_4)=k_1X_1+k_2X_2+k_3X_3$ for $k_1,k_2,k_3\in \Bbbk$, which we proved before, implies that $k_1X_1X_8+k_2X_2X_8+4X_4=0$. 
This is impossible since the elements $X_1X_8,X_2X_8$ and $X_4$ are linearly independent in $\HH^1(A)$ (see \cite{es zl}, (6.7)). 
The proposition thus follows. 
\end{proof}

\bibliographystyle{model1-num-names}

\end{document}